\documentclass[12pt]{article}
\usepackage{latexsym,amsmath,amssymb}

\newif\ifdviwin

\usepackage[english]{babel}
\usepackage{indentfirst}
\usepackage[mathscr]{eucal}
\usepackage{amssymb,amsmath,amsfonts}
\usepackage{fancybox,fancyhdr}
\usepackage{graphicx}
\usepackage[utf8]{inputenc}
\usepackage{float}
\usepackage{color}
\usepackage[export]{adjustbox}

\newif\ifdviwin

\dviwintrue

\def\cC{\mathcal{C}}

\def\cB{\mathcal{B}}

\def\m2r{\mathbb{M}^2\times\mathbb{R}}
\def\h2r{\mathbb{H}^2\times\mathbb{R}}

\let\8=\infty \let\0=\emptyset

\def\cte.{\mathop{\rm cte.}\nolimits}

\def\div{\mathop{\rm div }\nolimits}

\def\cosh{\mathop{\rm cosh }\nolimits}
\def\tanh{\mathop{\rm tanh }\nolimits}

\def\R{\mathbb{R}}

\def\cB{\mathcal{B}}

\def\cN{\mathcal{N}}

\def\D{\mathbb{D}}

\def\H{\mathbb{H}}
\def\S{\mathbb{S}}

\def\arctanh{\mathrm{arctanh}}

 \newtheorem{teo}{Theorem}[section]
  \newtheorem{defi}[teo]{Definition}
 \newtheorem{pro}[teo]{Proposition}
 
 \newtheorem{lem}[teo]{Lemma}

 \newtheorem{obs}[teo]{Observation}

 \newenvironment{proof}{\rm \trivlist \item[\hskip \labelsep{\it
      Proof}:]}{\nopagebreak \hfill $\Box$ \endtrivlist}

\numberwithin{equation}{section}

\begin{document}

\mbox{}\vspace{0.4cm}

\begin{center}
\rule{17cm}{1.5pt}\vspace{0.5cm}

\renewcommand{\thefootnote}{\,}
{\Large \bf Translating solitons of the mean curvature flow in the space $\h2r$ \footnote{\hspace{-.75cm} Mathematics Subject
Classification: 53A10, 53C42}}\\ \vspace{0.5cm} {\large Antonio Bueno}\\ \vspace{0.3cm} \rule{17cm}{1.5pt}
\end{center}
  \vspace{1cm}
Departamento de Geometría y Topología, Universidad de Granada,
E-18071 Granada, Spain. \\ e-mail: jabueno@ugr.es \vspace{0.3cm}

\begin{abstract}
In this paper we study the theory of translating solitons of the mean curvature flow of immersed surfaces in the product space $\h2r$. We relate this theory to the one of manifolds with density, and exploit this relation by regarding these translating solitons as minimal surfaces in a conformal metric space. Explicit examples of these surfaces are constructed, and we study the asymptotic behavior of the existing rotationally symmetric examples. Finally, we prove some uniqueness and non-existence theorems.
\end{abstract}
 
\section{Introduction}

Let $M$ be an orientable, immersed surface in the product space $\h2r$. We will say that $M$ is a \emph{translating soliton} if the mean curvature $H_M$ of $M$ satisfies at each $p\in M$
\begin{equation}\label{hyperbolicsoliton}
H_M(p)\eta_p=\partial_z^\bot.
\end{equation}
Here, $(\cdot)^\bot$ denotes the \emph{normal component}, $\partial_z$ is the unit vertical Killing vector field in $\h2r$ and $\eta$ is a unit normal vector field defined on $M$. Throughout this paper we will denote by $\langle\cdot,\cdot\rangle=g_{\H^2}+dz^2$ to the product metric in $\h2r$; here $g_{\H^2}$ is the metric in $\H^2$ of constant curvature $-1$. Notice that the above equation can be rewritten as
\begin{equation}\label{mediapresc}
H_M(p)=\langle\eta_p,\partial_z\rangle,
\end{equation} 
where the scalar quantity $\langle \eta, \partial_z\rangle$ is the \emph{angle function} of the surface $M$ computed with respect to $\eta$ and the $e_3$ direction, and will be denoted by $\nu(p):=\langle \eta_p, \partial_z\rangle$ for all $p\in M$. 

Our objective in this paper is to take as a starting point the well studied theory of \emph{translating solitons} of the mean curvature flow (MCF for short) in the Euclidean space $\R^3$, and use some of the known examples of translating solitons in $\h2r$ to obtain uniqueness and non-existence theorems, see \cite{CSS,Hu,HuSi,Il,MPGSHS,MSHS,Sm,SpXi} for relevant works regarding translating solitons of the MCF in $\R^3$. First, recall some basic notions about the MCF in $\R^3$. Let $\psi:M\rightarrow\R^3$ be an immersion of an orientable surface $M$ in the Euclidean space $\R^3$. Define by $\psi_t(\cdot)=\psi(\cdot,t):M\times [0,T)\rightarrow\R^3$ a smooth variation of $M$, where $T>0$. We say that the variation $\psi_t$ evolves by MCF if 
\begin{equation}\label{MCF}
\bigg(\frac{\partial\psi_t}{\partial t}(p,t)\bigg)^\bot=H_t(\psi_t(p))(\eta_t)_{\psi_t(p)},\ \forall p\in M,t\in [0,T),
\end{equation}
where $\eta_t:M_t\rightarrow\R^3$ is a unit normal vector field of $M_t=\psi_t(M)$ and $H_t$ is the mean curvature of the surface $M_t$ computed with respect to $\eta_t$.  A surface $M$ in $\R^3$ is a translating soliton of the MCF if it is a solution of Equation \eqref{MCF} for the particular variation given by Euclidean translations $\psi_t(p)=\psi(p)+tv$, where $v\in\R^3$ is a fixed vector named the \emph{translating vector}. As a matter of fact, $\eta_t=\eta$ and $H_t=H_M$, and thus Equation \eqref{MCF} reduces to
\begin{equation}\label{solitonMCF}
H_M(p)=\langle\eta_p,v\rangle.
\end{equation} 
In \cite{HuSi} the authors proved that translating solitons in $\R^3$ appear in the singularity theory of the MCF as the equation of the limit flow by a proper blow-up procedure near type II singular points. Since $\R^3$ is \emph{isotropic} and no direction at all is in some sense \emph{privileged}, after an Euclidean change of coordinates which leaves the problem invariant we may suppose that $v$ is the vertical vector $e_3$. Equation \eqref{solitonMCF} shows us that translating solitons of the MCF in $\R^3$ can be seen as a \emph{prescribed curvature problem}, only involving the measurement of the angle that makes a unit normal field defined on the surface with a unit Killing vector field in the space. 

Among the most recognized examples of translating solitons in $\R^3$, the ones invariant under the $SO(2)$ action of rotations around a fixed axis parallel to the direction of the flow have special interest in themselves. The complete, rotational translating solitons are classified as follows: there exists a complete, strictly convex translating soliton which is an entire graph over $\R^2$, called the \emph{bowl soliton}; and there exists a 1-parameter family of properly embedded annuli called the \emph{translating catenoids} or \emph{wing-like solitons}, see \cite{AlWu,CSS,MSHS,SpXi} for relevant works regarding asymptotic behavior at infinity as well as characterizations of these examples. Concretely, in \cite{CSS} the authors proved that the bowl soliton and the translating catenoids are \emph{asymptotic at infinity} when expressed as graphs outside a compact set.

In recent years, the space $\h2r$ has been considered as a major framework to extend the classical theory of minimal surfaces and non-vanishing constant mean curvature surfaces in the Euclidean space $\R^3$. Many geometers have focused on this space in the last years, developing a fruitful theory of immersed surfaces in $\h2r$. See \cite{NeRo1,NeRo2} for some remarkable works regarding this space. 

The structure of this paper is the following: in \textbf{Section \ref{examples}}, we will study the first properties of translating solitons in $\h2r$, taking as main motivation the well studied theory of translating solitons of the MCF in $\R^3$. We introduce the two models of the space $\h2r$ that we are going to work with. In Theorem \ref{car3} we characterize translating solitons in $\h2r$ as minimal surfaces in a conformal space, and in a density space. In particular, we show that these solitons are critical points for the weighted area functional, as introduced by Gromov in \cite{Gr}. This point of view of translating solitons as minimal surfaces allows us to prove the tangency principle in Theorem \ref{tangency}, and to solve the Dirichlet problem in Proposition \ref{dirichlet}.

The simplest examples of surfaces to study are those invariant under the 1-parameter action of isometries of $\h2r$. In \textbf{Section \ref{rotacionales}}, the considered uniparametric group of isometries are rotations around a vertical axis, and the examples arising are quite similar to the ones in the translating solitons of the MCF theory. These rotationally symmetric examples were constructed by E. Kocakusakli, M. A. Lawn and M. Ortega, see \cite{KoOr,LaOr} in the semi-Riemannian setting, and in particular in the spaces $\H^n\times\R$. Here we will prove the existence of such examples by using a phase space analysis in same fashion as in \cite{BGM}, where the authors studied immersed surfaces in $\R^3$ whose mean curvature is given as a prescribed function in the sphere $\S^2$ depending on its Gauss map. The main idea is that the ODE satisfied by the coordinates of the \emph{generating curve} of a rotationally symmetric translating soliton, can be expressed as a first order autonomous system. With these tools, in Section \ref{exbowl} we prove the existence of the \emph{bowl soliton}, and in Section \ref{excat} we prove the existence of the \emph{translating catenoids}, also called wing-like examples. Both the bowl soliton and the family of translating catenoids are the analogous to the rotationally symmetric translating solitons of the MCF in $\R^3$.

In \textbf{Section \ref{asintotico}}, motivated by the graphical computations of the rotationally symmetric solitons, we study the behavior at infinity of the  bowl soliton and the translating catenoids. In Lemma \ref{asin} we will obtain the behavior that a rotational, graphical translating soliton has when approaching to infinity. In particular, the  bowl soliton and the ends of each translating catenoid have the same asymptotic behavior when expressed as graphs outside a compact set.

Lastly, in \textbf{Section \ref{uniqueness}} we use the examples defined in Section \ref{examples}, their asymptotic behavior at infinity exposed in Section \ref{asintotico}, and Theorem \ref{tangency} to prove some uniqueness and non-existence theorems. Most of the theorems obtained in this Section, motivated by the thesis manuscript in \cite{Pe}, are proved by comparing with a proper translating soliton of the previously introduced, and then invoking Theorem \ref{tangency} in order to arrive to a contradiction. The main result here is Theorem \ref{unicidadbowl}, where we prove that an immersed translating soliton  with finite topology and one end which is $C^1$-asymptotic to the  bowl soliton has to be a vertical translation of the bowl. This similar characterization of the bowl soliton of the MCF in $\R^3$ was first obtained in \cite{MSHS}.
\vspace{1cm}

{\bf Acknowledgements} The author is grateful to the referee for helpful comments that highly improved the final version of the paper.

\section{Preliminaries on translating solitons in the product space $\h2r$}\label{examples}
Throughout this paper we will use two models of the hyperbolic plane $\H^2$:

\begin{itemize}
\item Let $\mathbb{L}^3$ denote the usual Lorentz-Minkowski flat space with global coordinates $(x_1,x_2,x_3)$ and endowed with the metric $dx_1^2+dx_2^2-dx_3^2$. The hyperbolic plane can be regarded as the hypercuadric defined as
$$
\H^2=\{(x_1,x_2,x_3)\in\mathbb{L}^3;\ x_1^2+x_2^2-x_3^2=-1,\ x_3>0\},
$$
endowed with the restriction of the ambient metric.

\item Consider the disk $\D=\{(x_1,x_2)\in\R^2;\ x_1^2+x_2^2<1\}$ endowed with the metric $ds^2=\lambda^2(dx_1^2+dx_2^2)$, where
\begin{equation}\label{conformalfact}
\lambda=\frac{2}{1-x_1^2-x_2^2}.
\end{equation}
Then, the space $(\D,ds^2)$ is isometric to $\H^2$ and its known as the \emph{Poincarè disk model} of $\H^2$. In this model we have global coordinates $(x_1,x_2,z)$, where $(x_1,x_2)\in\D$ and $z\in\R$, and a global orthonormal frame given by
$$
E_1=\frac{1}{\lambda}\partial_{x_1},\hspace{.5cm} E_2=\frac{1}{\lambda}\partial_{x_2},\hspace{.5cm} E_3=\partial_z.
$$
\end{itemize}

The product space $\h2r$ is defined as the Riemannian product of the hyperbolic plane $\H^2$ and the real line $\R$, endowed with the usual product metric which we will denote by $\langle\cdot,\cdot\rangle$.

In the space $\h2r$ there are defined the two usual projections $\pi_1:\h2r\rightarrow\H^2$ and $\pi_2:\h2r\rightarrow\R$. The \emph{height function} of $\h2r$ is defined to be the second projection $\pi_2$, and is commonly denoted as $h(p):=\pi_2(p)$ for all $p\in\h2r$. The gradient of the height function is a vertical, unit Killing vector field and is commonly denoted in the literature by $\partial_z$.

Let us point out two key properties that translating solitons in $\R^3$ satisfy. One of the main tools in this theory is the fact that they can be regarded as minimal surfaces in the conformal space $\big(\R^3,e^{x_3}g_{euc}\big)$, where $x_3$ stands for the third coordinate of a point and $g_{euc}$ is the usual Euclidean metric. The conformal metric $e^{x_3}g_{euc}$ is known in the literature as the \emph{Ilmanen metric}, see \cite{Il} for more details. Consequently, every translator in $\R^3$ is a minimal surface in $\big(\R^3,e^{x_3}g_{euc}\big)$ and vice versa. 

The theory of translating solitons in the Euclidean space is also related with the one of manifolds with density as follows: Let $(\cN,g,\phi)$ be a manifold with a density function $\phi\in C^\infty(\cN)$. For manifolds with density, Gromov \cite{Gr} defined the \emph{weighted mean curvature} of an oriented hypersurface $M\subset(\cN,g,\phi)$ by
\begin{equation}\label{weightedmc}
H_\phi:=H_M-g(\nabla\phi,\eta),
\end{equation}
where $H_M$ is the mean curvature of $M$ in $(\cN,g)$, $\eta$ is a unit normal vector field along $M$ and $\nabla$ is the gradient computed in the ambient space $(\cN,g)$, see also \cite{BCMR}. 

For the particular case when we consider the weighted space $\big(\R^3,g_{euc},x_3\big)$, from \eqref{weightedmc} we obtain
$$
H_{x_3}=H-g_{euc}(\eta,\nabla x_3)=H-g_{euc}(\eta,e_3).
$$
Thus, an immersed surface in $\R^3$ is a translating soliton if and only if it is a minimal surface in $\R^3$ measured with the density $x_3$.

The next theorem proves that the translating solitons in $\h2r$ inherits these same properties.
\begin{teo}\label{car3}
Let $M$ be an immersed surface in $\h2r$. Then, are equivalent:

\begin{itemize}
\item[1.] The surface $M$ is a translating soliton in $\h2r$.

\item[2.] The surface $M$ is minimal in the conformal space $\big(\h2r,e^{h}\langle\cdot,\cdot\rangle\big)$.

\item[3.] The surface $M$ is weighted minimal in the density space $\big(\h2r,\langle\cdot,\cdot\rangle,h\big)$.
\end{itemize}
\end{teo}

\begin{proof}
$\underline{1.\Longleftrightarrow 2.}$ A wide-known formula states that given a three dimensional Riemannian manifold $(\cN,g)$ and a conformal metric $\overline{g}=e^{2\phi}g$, where $\phi$ is a smooth function on $\cN$, the mean curvatures $\overline{H}_M$ and $H_M$ of $M$ with respect to the metrics $\overline{g}$ and $g$ respectively, are related by the formula
$$
\overline{H}_M=e^{-\phi}\big(H_M-2g(\nabla\phi,\eta)\big),
$$
where $\nabla$ is the gradient operator with respect to the metric $g$.

In our situation, the conformal factor is just the exponential of the height function $h$. As the gradient of the height function in $\h2r$ is no other than the vertical Killing vector field $\partial_z$, we obtain
$$
\overline{H}_M=e^{-h/2}\big(H_M-\langle\partial_z,\eta)\rangle\big)=e^{-h/2}\big(H_M-\nu\big).
$$
Thus, $M$ is a translating soliton if and only if the mean curvature $\overline{H}_M$ vanishes identically, proving the equivalence between the first items.
\vspace{.5cm}

$\underline{1.\Longleftrightarrow 3.}$ From Equation \eqref{weightedmc} for the density space $\big(\h2r,\langle\cdot,\cdot\rangle,h\big)$, we obtain
$$
H_h=H_M-\langle\partial_z,\eta\rangle=H_M-\nu.
$$
This proves that $M$ is a translating soliton in $(\h2r,\langle\cdot,\cdot\rangle)$ if and only if $M$ is weighted minimal in $\big(\h2r,\langle\cdot,\cdot\rangle,h\big)$.
\end{proof}

The importance of Item $3$ in the previous theorem is that translating solitons can be characterized as critical points of the weighted area functional in the following way: Let $(\cN,g,\phi)$ a density manifold. For a measurable subset $\Omega\subset\cN$ with boundary $M=\partial\Omega$ and inward unit normal $\eta$, we can define the \emph{weighted area} of $M$ as
$$
A_\phi(M)=\int_M e^\phi dv_M,
$$
where $dv_M$ stands for the area element with respect to the metric $g$. Consider a compactly supported variation $\{\Psi_t\}$ of a immersed surface $M$ with $\Psi'(0)=V+\omega\eta$, where $V$ is a tangent vector field along $M$ and $\omega$ is a smooth function with compact support on $M$. By Bayle's variational formula in \cite{Ba}, we have
$$
\frac{d}{dt}\bigg|_{t=0}A_\phi(\Psi_t(M))=\int_M H_\phi\omega e^\phi dv_M,
$$
and thus weighted minimal surfaces in density spaces are critical points of the weighted area functional.

In particular, Item 3 in Theorem \ref{car3} ensures us that \emph{\textbf{translating solitons are critical points for the weighted area functional under compactly supported variations.}}

The minimality of a translating soliton in the conformal space $(\h2r,e^{h}\langle\cdot,\cdot\rangle)$ given by Item $2$ in Theorem \ref{car3} allows us to formulate the \emph{tangency principle}, which resembles us to the case of minimal surfaces in $\R^3$ and is just a consequence of the maximum principle for elliptic PDE's due to Hopf.

\begin{teo}[Tangency principle]\label{tangency} Let $M_1$ and $M_2$ be two connected translating solitons in $\h2r$ with possibly non-empty boundaries $\partial M_1,\ \partial M_2$. Suppose that one of the following statements holds
\begin{itemize}
\item There exists $p\in int(M_1)\cap int (M_2)$ with $(\eta_1)_p=(\eta_2)_p$, where $\eta_i:M_i\rightarrow\S^2$ is the unit normal of $M_i$, respectively. 

\item There exists $p\in\partial M_1\cap\partial M_2$ with $(\eta_1)_p=(\eta_2)_p$ and $(\xi_1)_p=(\xi_2)_p$, where $\xi_i$ is the interior unit conormal of $\partial M_i$.
\end{itemize}

Assume that $M_1$ lies locally around $p$ at one side of $M_2$. Then, in either situation, both surfaces agree in a neighbourhood of $p$. Moreover, if both surfaces $M_i$ are complete, then $M_1=M_2$.
\end{teo}

We also focus our attention on solving the \emph{Dirichlet problem} for graphical translating solitons. The next result is a consequence of Theorem 1.1 in \cite{CHH}, and gives conditions for the existence of graphical translating solitons in $\h2r$.

\begin{pro}\label{dirichlet}
Let $\Omega\subset\H^2$ be a bounded $C^2$ domain with $C^{2,\alpha}$ boundary, and consider $\varphi\in C^{2,\alpha}(\partial\Omega)$ for $\alpha\in (0,1)$. Suppose that $H_{\partial\Omega}\geq 2$, where $H_{\partial\Omega}$ stands for the inward curvature of $\partial\Omega$. Then, the Dirichlet problem
\begin{equation}\label{divergencia}
\left\{\begin{array}{ll}
2H_M=\displaystyle{\frac{2}{\sqrt{1+|\nabla u|^2}}}=\div\Bigg(\frac{\nabla u}{\sqrt{1+|\nabla u|^2}}\Bigg) & in\ \Omega\\
u=\varphi & on\ \partial\Omega
\end{array}\right.
\end{equation}
has a unique solution $u\in C^{2,\alpha}(\overline{\Omega})$.
\end{pro}

\begin{proof}
We will check that our hypothesis agree with the hypothesis in Theorem 1.1 in \cite{CHH}, which is formulated in a more general setting; there, $\Omega$ is an open subset of a complete, non-compact manifold $M$, and Equation \eqref{divergencia} has the expression
$$
\div\Bigg(\frac{\nabla u}{\sqrt{1+|\nabla u|^2}}\Bigg)=\langle\overline{\nabla} f,\eta\rangle,
$$
where $\eta$ is the unit normal of the graph, $f$ is a smooth function defined in the product manifold $M\times\R$, and $\overline{\nabla}$ is the gradient operator computed with respect the product metric. In our setting, the function $f$ defined on $\h2r$ is just $f(p)=2h(p)$, where $h$ denotes as usual the height function of a point $p\in\h2r$, and thus $\overline{\nabla} f=2\partial_z$. In the same spirit as in Theorem 1.1, we define $F=\underset{\overline{\Omega\times\R}}{\sup}|\overline{\nabla} f|$. In our particular case, $F=2$. Now, for applying Theorem 1.1, three conditions must hold:
\begin{itemize}
\item[1.] $F<\infty$, which is trivial in our case.
\item[2.] $Ric_\Omega\geq -F$. In our case, $\Omega\subset\H^2$, which has constant curvature equal to $-1$. The Ricci curvature of the hyperbolic plane is equal to $-2$, and thus this result also holds trivially.
\item[3.] $H_{\partial\Omega}\geq F$. This is just the hypothesis stated at the formulation of Proposition \ref{dirichlet}.
\end{itemize}
In this situation, Theorem 1.1 in \cite{CHH} ensures that there exists a function $u\in C^{2,\alpha}(\overline{\Omega})$, such that the graph defined by $u$ solves Equation \eqref{divergencia}. This completes the proof of Proposition \ref{dirichlet}.
\end{proof}

To end this section, we will give the first examples of solutions of Equation \eqref{hyperbolicsoliton}, which are minimal surfaces of $\h2r$. If a translating soliton is also a minimal surface, then the translating vector $\partial_z$ that defines the movement of the translating soliton must satisfy
\begin{equation}\label{hypvertical}
\partial_z^\bot=H_M(p)\eta_p=0,
\end{equation}
that is, $\partial_z$ has to be tangential to the translator $M$ at each $p\in M$. This happens for vertical planes $\gamma\times\R$, where $\gamma\subset\H^2$ is a geodesic, which are minimal surfaces of $\h2r$ everywhere tangential to $\partial_z$, and thus translating solitons. 

\section{Rotationally symmetric translating solitons}\label{rotacionales}

This section is devoted to the study of translating solitons which are invariant under the isometric $SO(2)$-action of rotations around a vertical axis. These examples were already obtained in \cite{KoOr,LaOr} for translating solitons immersed in semi-Riemannian manifolds. An alternative proof will be given in this paper, and the existence, uniqueness and properties of these rotationally symmetric translators will be analysed by means of a phase space study, inspired by the ideas developed in Section 3 in \cite{BGM}.

\textbf{Throughout this section, the model used for the space $\H^2$ will be the Lorentz-Minkowski hyperboloid in $\mathbb{L}^3$.}

After an ambient translation, we may suppose that the vertical axis is the one passing through the origin. Let $\alpha(t)=(\sinh r(t),0,\cosh r(t),w(t)),\ t\in I\subset \R$ be an arc-length parametrized curve in the space $\h2r$. In this situation, we can make $\alpha$ rotate around the vertical axis passing through the origin  under the isometric $SO(2)$-action of a circle $\phi(s)=(\cos s,\sin s)$. Bearing this in mind, the parametrization given by
\begin{equation}\label{parrot}
\psi(t,\theta)=(\sinh r(t)\cos \theta,\sinh r(t)\sin \theta,\cosh r(t),w(t))
\end{equation}
generates an immersed surface $M$, rotationally symmetric with respect to the vertical axis passing through the origin. With this parametrization the angle function is given, up to a change of the orientation, by $\nu(\psi(t,\theta))=r'(t)$. The principal curvatures of $M$ at each $\psi(t,s)$ are given by
\begin{equation}\label{curvaturasprinc}
\kappa_1=\kappa_\alpha=r'(t)w''(t)-r''(t)w'(t),\ \ \ \ \ \ \ \kappa_2=w'(t)\coth r(t),
\end{equation}
where $\kappa_\alpha$ is the geodesic curvature of $\alpha(t)$. The mean curvature of a rotationally symmetric surface in $\h2r$ has the expression
$$
2H_M=r'(t)w''(t)-r''(t)w'(t)+w'(t)\coth r(t).
$$
By hypothesis, $M$ is a translating soliton and thus $H_M(\psi(t,\theta))=r'(t)$. This implies that the coordinates $r(t),w(t)$ of an arc-length parametrized curve, generating a rotationally symmetric translating soliton given by Equation \eqref{parrot}, satisfy the system
\begin{equation}\label{ODErot}
 \left\{\begin{array}{rll}
 r'(t)&=&  \cos\theta(t)\\
 w'(t)&=& \sin\theta(t)\\
 \theta'(t)&=& 2\cos\theta(t)-\sin\theta(t)\coth r(t).
\end{array}
\right.
\end{equation}
From now on, we will suppress the dependence of the variable $t$ and just write $r\equiv r(t)$, and so on. The arc-length parametrized condition $r'^2+w'^2=1$ implies that the function $r$ is a solution of the autonomous second order ODE
\begin{equation}\label{ODErot2}
r''=(1-r'^2)\coth r-2\varepsilon r'\sqrt{1-r'^2},\ \ \ \ \varepsilon=\mathrm{sign}(w'),
\end{equation}
on every subinterval $J\subset I$ where $w'\neq 0$.

The change $r'=y$ transforms \eqref{ODErot2} into the first order autonomous system
\begin{equation}\label{1ordersys}
\left(\begin{array}{c}
r'\\
y'
\end{array}\right)=\left(\begin{array}{c}
y\\
(1-y^2)\coth r -2\varepsilon y\sqrt{1-y^2}
\end{array}\right)=F(r,y).
\end{equation}

We define the \emph{phase space} of \eqref{1ordersys} as the half-strip $\Theta_\varepsilon:=(0,\8)\times (-1,1)$, with coordinates $(r,y)$ denoting, respectively, the distance to the rotation axis and the angle function. It will be also useful to define the sets $\Theta_\varepsilon^+:=(\Theta_\varepsilon\cap\{y>0\})$ and $\Theta_\varepsilon^-:=(\Theta_\varepsilon\cap\{y<0\})$. The \emph{equilibrium points}, if they exist, correspond to points at constant distance to the axis of rotation. They can be characterized by the fact that $F(r_0,y_0)=0$. In this case no equilibrium points exist, and thus there are no translating solitons that can be considered as rotational vertical cylinders. 

A straightforward consequence of the uniqueness of the Cauchy problem is that the orbits $\gamma(t):=(r(t),y(t))$ are a foliation by \emph{\textbf{regular proper $C^2$ curves}} of $\Theta_\varepsilon$. This properness condition will be applied throughout this paper, and should be interpreted as follows: any orbit $\gamma(t)$ cannot have as endpoint a finite point of the form $(x_0,y_0)$ with $x_0\neq 0$ and $y_0\neq\pm 1$, since at these points Equation \eqref{1ordersys} has local existence and uniqueness, and thus any orbit around a point $(x_0,y_0)$ can be extended.

This properness condition implies that any orbit $\gamma(t)$ is a maximal curve inside $\Theta_\varepsilon$ which has its endpoints at the boundary $\overline{\Theta}_\varepsilon=\{0\}\times\{1,-1\}$.

The points in $\Theta_\varepsilon$ where $y'=0$ are those lying at the horizontal graph
\begin{equation}
r=\Gamma_\varepsilon(y)=\arctanh\bigg(\frac{\sqrt{1-y^2}}{2\varepsilon y}\bigg).
\end{equation}

We will denote by $\Gamma_\varepsilon$ the intersection $\Theta_\varepsilon\cap\Gamma_\varepsilon(y)$. It is immediate to observe that the values $t\in J$ where the profile curve $\alpha$ has vanishing geodesic curvature are those where $y'=0$, i.e. the points where $(r(t),y(t))\in\Gamma_\varepsilon$.

Notice that, as the function $\arctanh$ is defined only for values lying in the interval $(-1,1)$, $\Gamma_\varepsilon(y)$ is only defined for values $y$ satisfying the bound
$$
-1<\frac{\sqrt{1-y^2}}{2y}<1\Longleftrightarrow |y|>\frac{1}{\sqrt{5}}.
$$
That is, the curve $\Gamma_\varepsilon$ has an asymptote at the lines $y=\pm 1/\sqrt{5}$. As $\Gamma_\varepsilon$ only appears at $\Theta_\varepsilon$ when $\varepsilon y\geq 0$, then for $\varepsilon=1$ (resp. $\varepsilon=-1$) $\Gamma_1$ only appears at $\Theta_1$ for $y\in(1/\sqrt{5},1]$ (resp. only appears at $\Theta_{-1}$ for $y\in [-1, -1/\sqrt{5})$). This implies that $\Gamma_\varepsilon$ and the axis $y=0$ divide $\Theta_\varepsilon$ into three connected components where both $r'$ and $y'$ are monotonous and $\alpha$ has non-vanishing geodesic curvature. It will be useful for the sake of clarity to name each of these \emph{monotonicity regions}: we define $\Theta_\varepsilon^+=\Theta_\varepsilon\cap\{y>0\}$ and $\Theta_\varepsilon^-=\Theta_\varepsilon\cap\{y<0\}$. When $\varepsilon=1$, then $\Gamma_1$ is contained entirely in $\Theta_1^+$. We define
\begin{equation}\label{regionesmonotoniamas}
\begin{array}{l}
\Lambda_1^-=\big(\Theta_1^+\cap\{y\leq 1/\sqrt{5}\}\big)\cup\{(r,y);\ y>1/\sqrt{5},\ r<\Gamma_1(y)\}, \vspace{.3cm}\\

\Lambda_1^+=\{(r,y);\ y>1/\sqrt{5},\ r>\Gamma_1(y)\}.
\end{array}
\end{equation}
which are, along with $\Theta_1^-$ the three monotonicity regions in $\Theta_1$, see Fig. \ref{fases}, left. Likewise, if $\varepsilon=-1$ then $\Gamma_{-1}$ is contained in $\Theta^-_{-1}$. Now we define 
\begin{equation}\label{regionesmonotoniamenos}
\begin{array}{l}
\Lambda_{-1}^-=\big(\Theta_{-1}^-\cap\{y\geq 1/\sqrt{5}\}\big)\cup\{(r,y);\ y<-1/\sqrt{5},\ r<\Gamma_1(y)\}, \vspace{.3cm}\\

\Lambda_{-1}^+=\{(r,y);\ y<-1/\sqrt{5},\ r>\Gamma_1(y)\}.
\end{array}
\end{equation}
In this situation the three monotonicity regions of $\Theta_{-1}$ are $\Theta_{-1}^+,\ \Lambda_{-1}^-$ and $\Lambda_{-1}^+$, see Fig. \ref{fases}, right.
\begin{figure}[H]
\hspace{-1.3cm}
\includegraphics[width=0.55\textwidth]{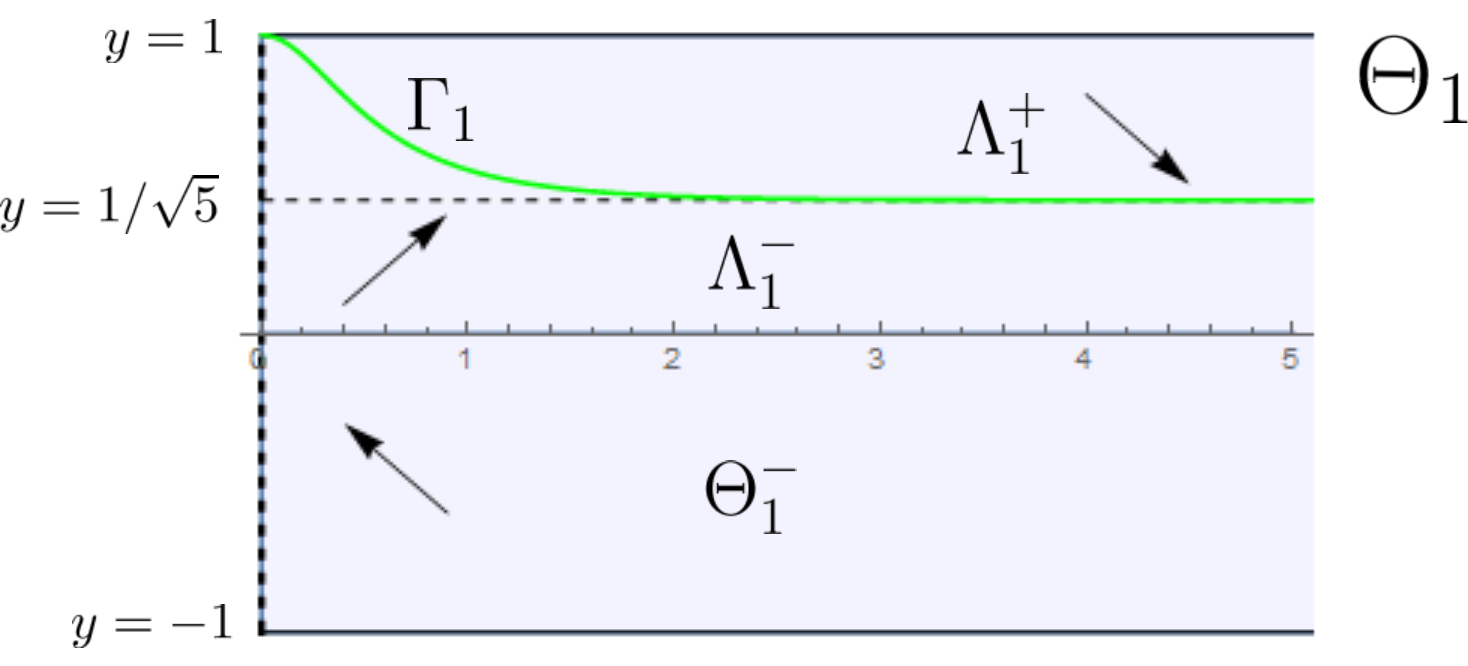} \includegraphics[width=0.6\textwidth]{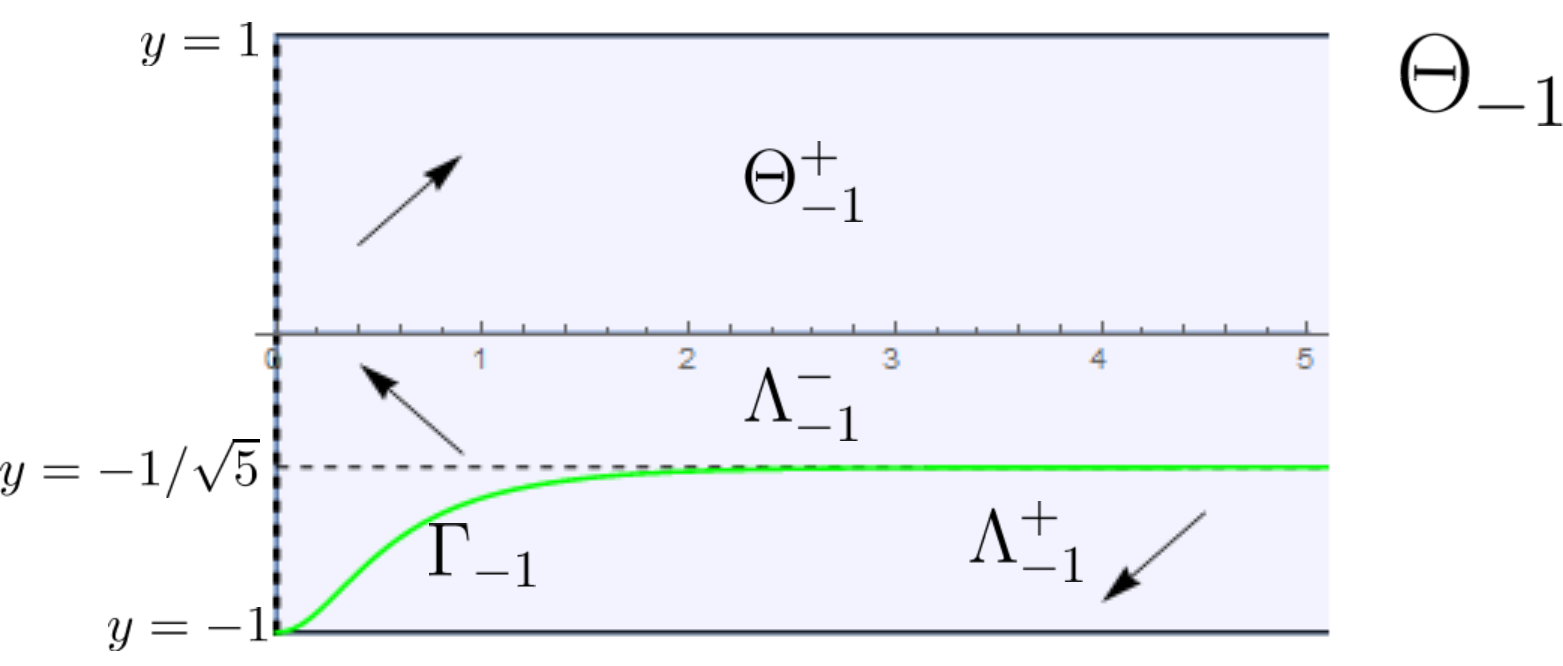}
\caption{Left: phase space $\Theta_1$. Right: phase space $\Theta_{-1}$. The monotonicity regions described in Equation \eqref{regionesmonotoniamas} and \eqref{regionesmonotoniamenos} are shown in each phase space.}
\label{fases}
\end{figure}
We should emphasize that the signs of the principal curvatures given by Equation \eqref{curvaturasprinc} at each point $\alpha(t)$ are given by
\begin{equation}\label{signk}
\mathrm{sign}(\kappa_1)=\mathrm{sign}(-\varepsilon y'(t)),\\\\\ \mathrm{sign}(\kappa_2)=\mathrm{sign}(\varepsilon).
\end{equation}
In each of these monotonicity regions we can view the orbits as functions $y=y(r)$ wherever possible, i.e. at points with $y\neq 0$, and thus we have
\begin{equation}\label{yfuncx}
y\frac{dy}{dr}=(1-y^2)\coth r-2\varepsilon y\sqrt{1-y^2}.
\end{equation}
In particular, in each monotonicity region the sign of $yy'$ is constant. As a consequence, the signs of $y_0$ and $r_0-\Gamma_\varepsilon(r_0)$ (for $y_0\neq 0$) determine the behavior of the orbit of \eqref{1ordersys} seen as a function $(r,y(r))$ in each component. The possible behaviors are summarized in the following Lemma:

\begin{lem}\label{comportamiento}
In the above setting, for any $(r_0,y_0)\in\Theta_\varepsilon$ with $y_0\neq 0$, the following properties hold:

\begin{itemize}
\item If $r_0>\Gamma_\varepsilon(y_0)$ (resp. $r_0<\Gamma_\varepsilon(y_0)$) and $y_0>0$, then $y(r)$ is strictly decreasing (resp. increasing), at $r_0$.

\item If $r_0>\Gamma_\varepsilon(y_0)$ (resp. $r_0<\Gamma_\varepsilon(y_0)$) and $y_0<0$, then $y(r)$ is strictly increasing (resp. decreasing), at $r_0$.

\item If $y_0=0$, then the orbit passing through $(r_0,0)$ is orthogonal to the $r$-axis.

\item If $r_0=\Gamma_\varepsilon(y_0)$, then $y'(r_0)=0$ and $y(r)$ has a local extremum at $r_0$.
\end{itemize}
\end{lem}

For any $(x_0,y_0)\in\Theta_\varepsilon^1$ we ensure the existence and uniqueness of the Cauchy problem of an orbit passing through $(x_0,y_0)$ that is a solution of system \eqref{1ordersys}. However, Equation \eqref{1ordersys} has a singularity at the points with $x_0=0$, and thus we cannot apply the existence and uniqueness of the Cauchy problem in order to guarantee the existence of an orbit having as endpoints either $(0,\pm 1)$. 

To overcome this difficulty we may solve the Dirichlet problem by Proposition \ref{dirichlet} in order to ensure the existence of a translating soliton in $\h2r$ which is rotational around the vertical axis passing through the origin and that meets this axis orthogonally at some point.

\begin{lem}\label{existeeje}
There exists a disk $\Omega\subset\H^2$ centered at the origin of $\H^2$ and a function $u:\Omega\rightarrow\R$ such that the surface defined by $M=\mathrm{graph}(u)$ is a translating soliton in $\h2r$ which is rotationally symmetric with respect to the vertical axis passing through the origin and that meets this axis in an orthogonal way at some $p\in M$.

Moreover, $M$ is unique among all the graphical translating solitons over $\Omega$ with constant Dirichlet data.
\end{lem}

\begin{proof}
We will expose the argument for \emph{upwards-oriented} graphs, since it is similar to \emph{downwards-oriented} graphs.

By Proposition \ref{dirichlet}, we can solve the Dirichlet problem in Equation \eqref{divergencia} for upwards-oriented graphs in a small enough disk $\Omega\subset\H^2$ centred at the origin with constant Dirichlet data on the boundary, obtaining a $C^{2,\alpha}$ function $u:\Omega\rightarrow\R$ that solves Equation \eqref{divergencia}.

Let us define $M:=\mathrm{graph}(u)$. As the mean curvature $H_M$ is given by the angle function, and it is rotationally symmetric, the translating soliton $M$ has the same symmetries as the prescribed function and thus $M$ is a rotational surface. The uniqueness of $M$ comes from the maximum principle, as the divergence equation is invariant up to additive constants.
\end{proof}

\subsection{The bowl soliton}\label{exbowl}

The following theorem proves the existence of the analogous to the bowl soliton in $\R^3$.
\begin{teo}\label{teobowl}
There exists an upwards-oriented, rotational translating soliton in $\h2r$ that is an entire, vertical graph.
\end{teo}

\begin{proof}
According to Lemma \ref{existeeje}, we ensure the existence of an upwards-oriented, rotational translating soliton $M$, generated by rotating an arc-length parametrized curve $\alpha(t)$ which is solution of \eqref{ODErot}. As $M$ is upwards-oriented, then at $p_0=M\cap l$, where $l$ is the vertical line passing through the origin, we have $H_M(p_0)=\nu(p_0)=1>0$. By the mean curvature comparison principle, the height function $w(t)$ of $\alpha(t)$ satisfies $w'(t)>0$, for $t>0$ close enough to zero, and thus the orbit $\gamma(t):=(r(t),y(t))$ starts at the point $(0,1)$ in $\Theta_1$ for $t>0$ small enough. Moreover, for $t>0$ small enough the curve $\alpha(t)$ has positive geodesic curvature, and thus the orbit $\gamma$ lies in $\Lambda^+$ for points near to $(0,1)$ in $\Theta_1$. The monotonicity properties imply that the whole $\gamma$ is contained in $\Lambda^+$ and $\kappa_{\alpha(t)}>0$ for all $t>0$. By monotonicity and properness, we can see $\gamma$ as a graph $y=f(r)$, for a certain $f\in C^2([0,\infty))$ satisfying $f(0)=1,\ f'(0)=0,\ f(r)>1/\sqrt{5}$ and $f'(r)<0$ for all $r>0$, see Fig. \ref{bowlperfil}. This implies that the translating soliton $M$ generated by rotating $\alpha$ with respect to the axis $l$, is an entire, vertical graph in $\h2r$, concluding the proof. 

\begin{figure}[H]
\centering
\includegraphics[width=0.6\textwidth]{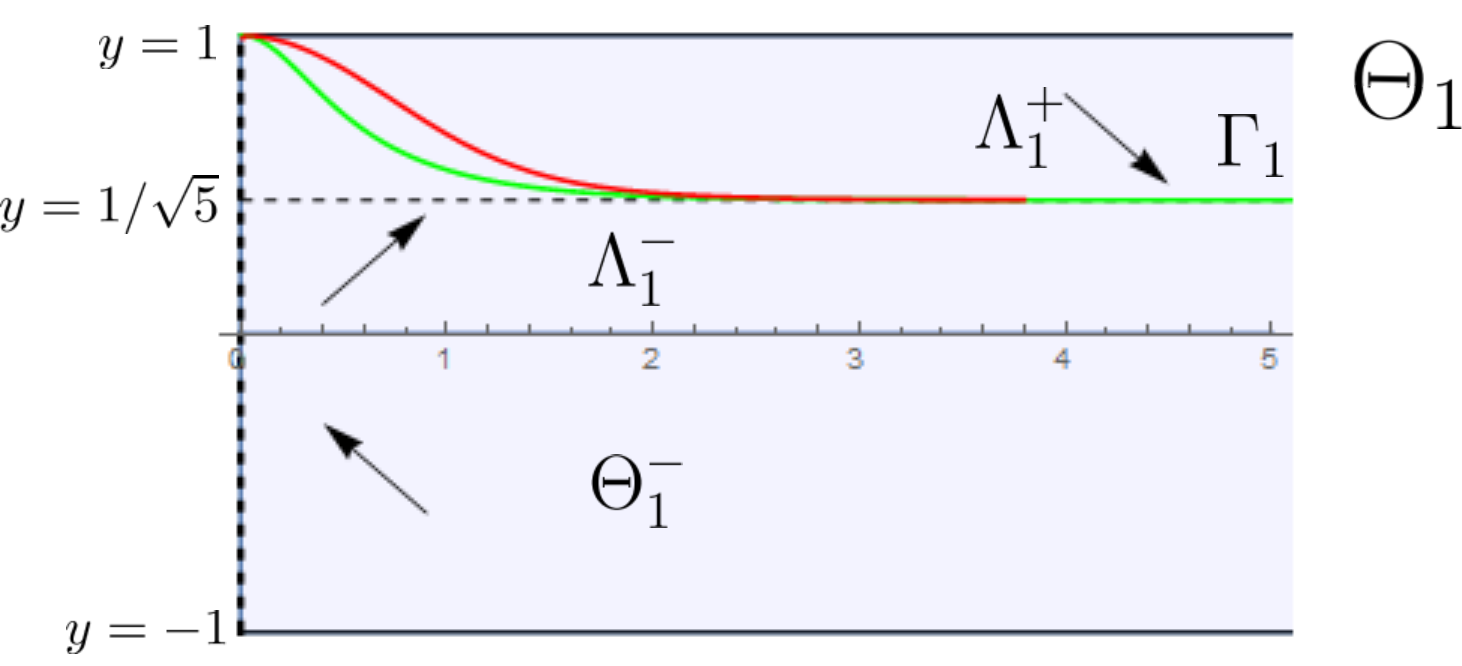}\ \ \ \ \ \ \ \ \ \ \ \includegraphics[width=0.17\textwidth]{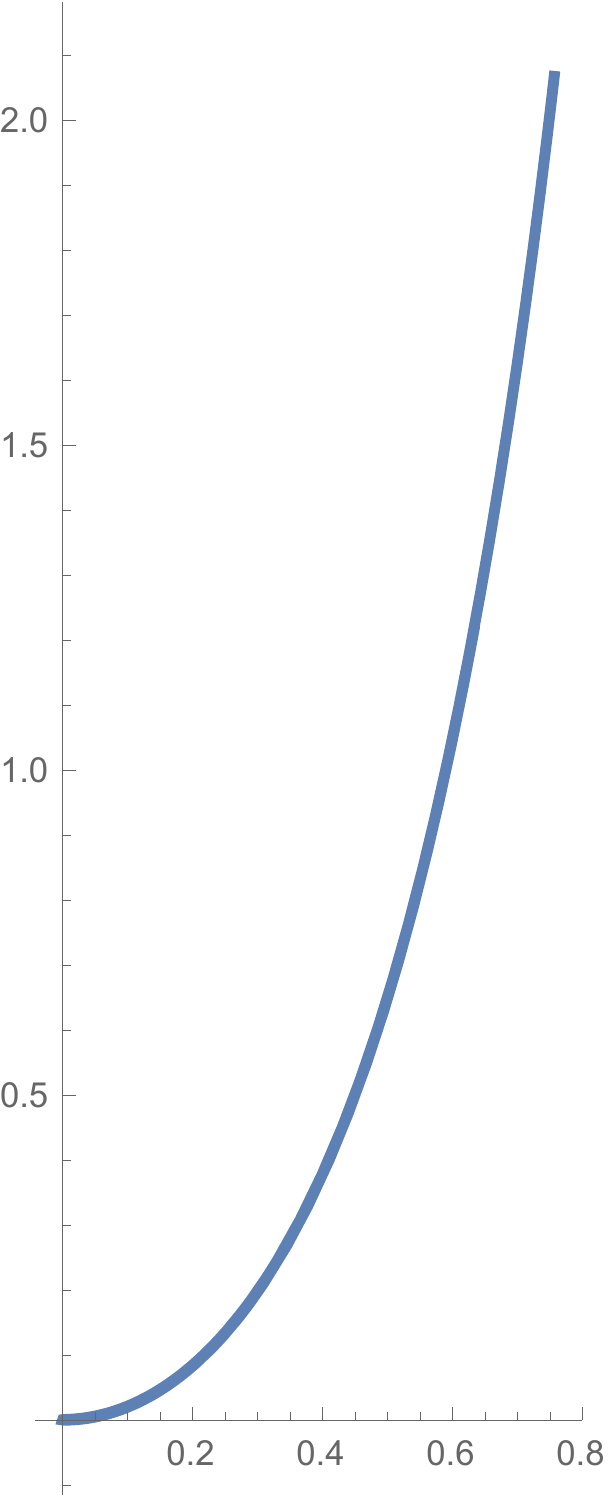}
\caption{Left: the phase space with the solution corresponded to the  bowl soliton plotted in red. Right: the profile of the bowl soliton in the model $\D\times\R$. Here we just plotted a compact piece of the  bowl soliton.}\label{bowlperfil}
\end{figure}
\end{proof}

This entire graph is called the \emph{bowl soliton}, and will be denoted throughout this paper by $\cB$ (see Fig. \ref{bowl}). The \emph{vertex} is the lowest point of $\cB$, which is also the unique point in $\cB$ that intersects the axis of rotation.

\begin{figure}[H]
\centering
\includegraphics[width=0.45\textwidth]{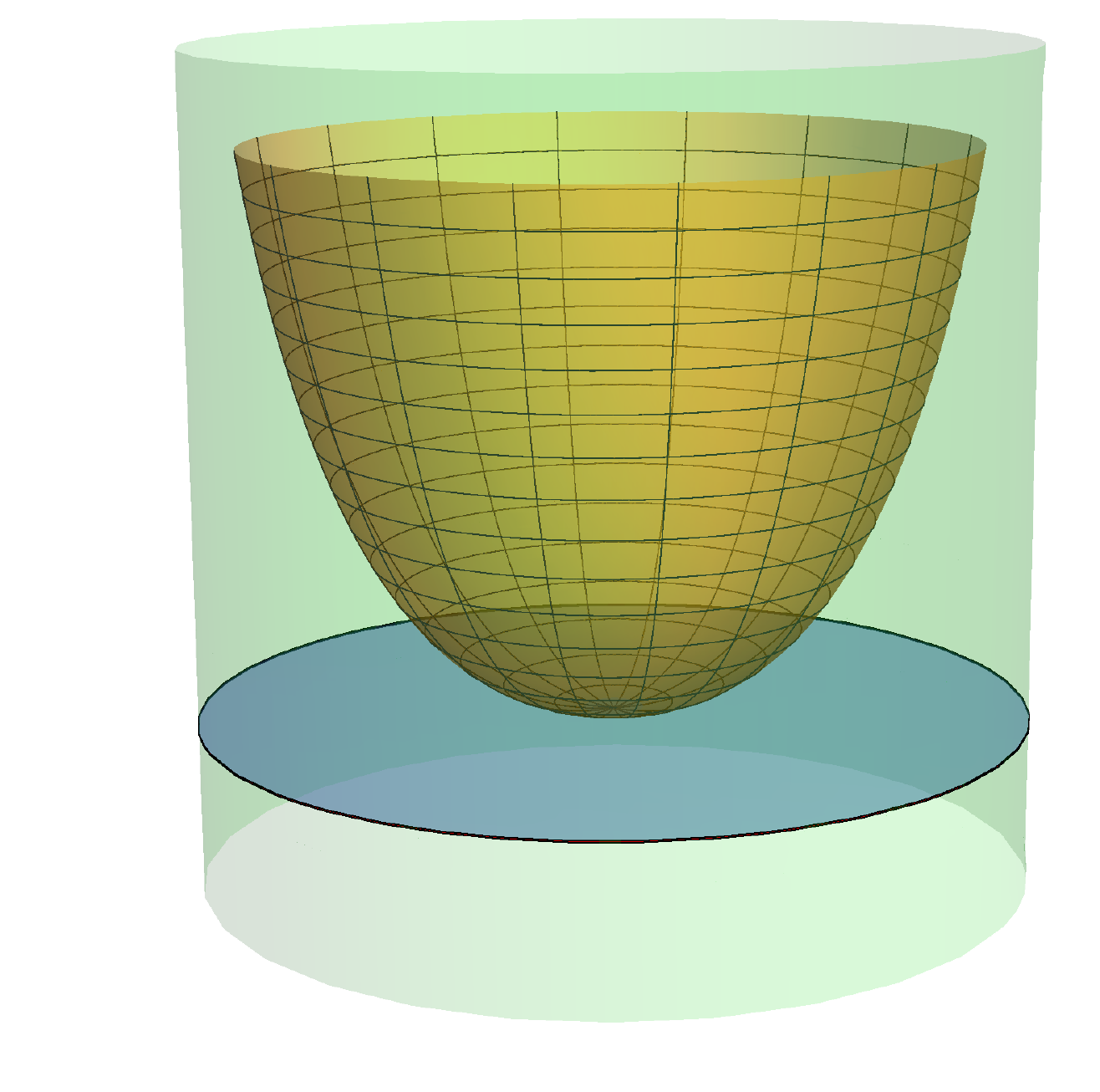}
\caption{The  bowl soliton in the Poincaré model $\D\times\R$.}\label{bowl}
\end{figure}

The main difference with the \emph{bowl soliton} in the theory of translating solitons in $\R^3$ is that although the bowl soliton has angle function tending to zero (and thus mean curvature tending to zero), here the  bowl soliton has angle function tending to $1/\sqrt{5}$, and thus the mean curvature at infinity is non-zero. However, in the space $\h2r$ no constant mean curvature spheres exist for values of the mean curvature $H\leq 1/2$. In particular, this behavior of $\cB$ does not contradict the mean curvature comparison theorem for constant mean curvatures spheres whose mean curvature approach to $1/2$.

\subsection{A one parameter family of immersed annuli: the translating catenoids}\label{excat}

The following theorem proves the existence of the analogous to the wing-like catenoids in $\R^3$.

\begin{teo}\label{catenoids}
There exists a one parameter family of properly immersed translating solitons, each one with the topology of an annulus. Each end of the annulus points to the $\partial_z$ direction, and is a vertical graph outside a compact set. These examples, denoted by $\{\cC_r\}_{r>0}$, are called the translating catenoids, or wing-like solutions.
\end{teo}

\begin{proof}
Let $M$ be the rotational translating soliton in $\h2r$ generated by the rotation of an arc-length parametrized curve $\alpha(t)$ given by Equation \eqref{parrot}, with initial conditions $r(0)=r_0,\ w'(0)=1$, for an arbitrary $r_0>0$. The orbit $\gamma=(r(t),y(t))$ passing through $(r_0,0)$ belongs to the phase space $\Theta_1$ for $t$ close enough to zero, i.e. $\varepsilon=1$ in \eqref{1ordersys}. In this situation, we know that there are three monotonicity regions in $\Theta_1$. For $t>0$ small enough, $\gamma$ stays in $\Lambda_1^-$, and by Lemma \ref{comportamiento} we can see the second coordinate of $\gamma$, $y(t)$, as an increasing function $y(r)$ until $\gamma$ intersects $\Gamma_1$, where $y(r)$ attains a maximum. Then, $\gamma$ lies inside $\Lambda_1^+$ and stays at it, and the coordinate $y(t)$ can be seen as a decreasing function $y(r)$ converging to $y=1/\sqrt{5}$, see Fig. \ref{fasecat1}, left. With this procedure we obtain the first component $M^+$, which is a graph over the exterior of the disk $D(0,r_0)$ of $\H^2$; indeed, the only point with $y'(0)$, i.e. with vertical tangency, occurs at $t=0$. This component has the topology of $\S^1\times[0,\infty)$, and $\S^1\times\{0\}$ is just the circumference at $t=0$. The height function $w(t)$ satisfies $w'(t)>0$ for every $t>0$. If we denote $t_0>0$ to the instant where $\gamma$ intersects $\Gamma_1$, then as $\gamma\subset\Lambda_1^+$ for all $t>t_0$, we conclude hat $M^+$ is a graph outside a compact set. By properness, the height of $M^+$ is unbounded as $t\rightarrow\infty$.

\begin{figure}[H]
\hspace*{-1.2cm}
\includegraphics[width=0.55\textwidth]{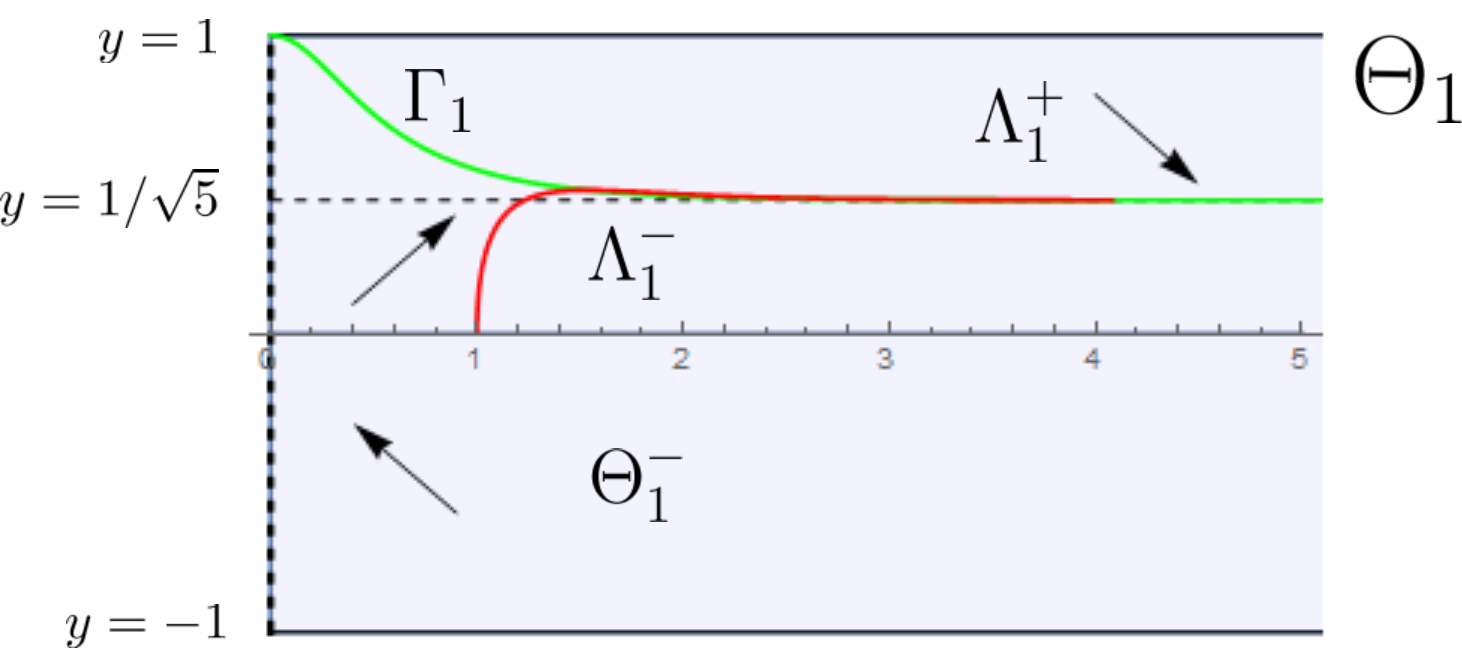}\ \ \  \includegraphics[width=0.55\textwidth]{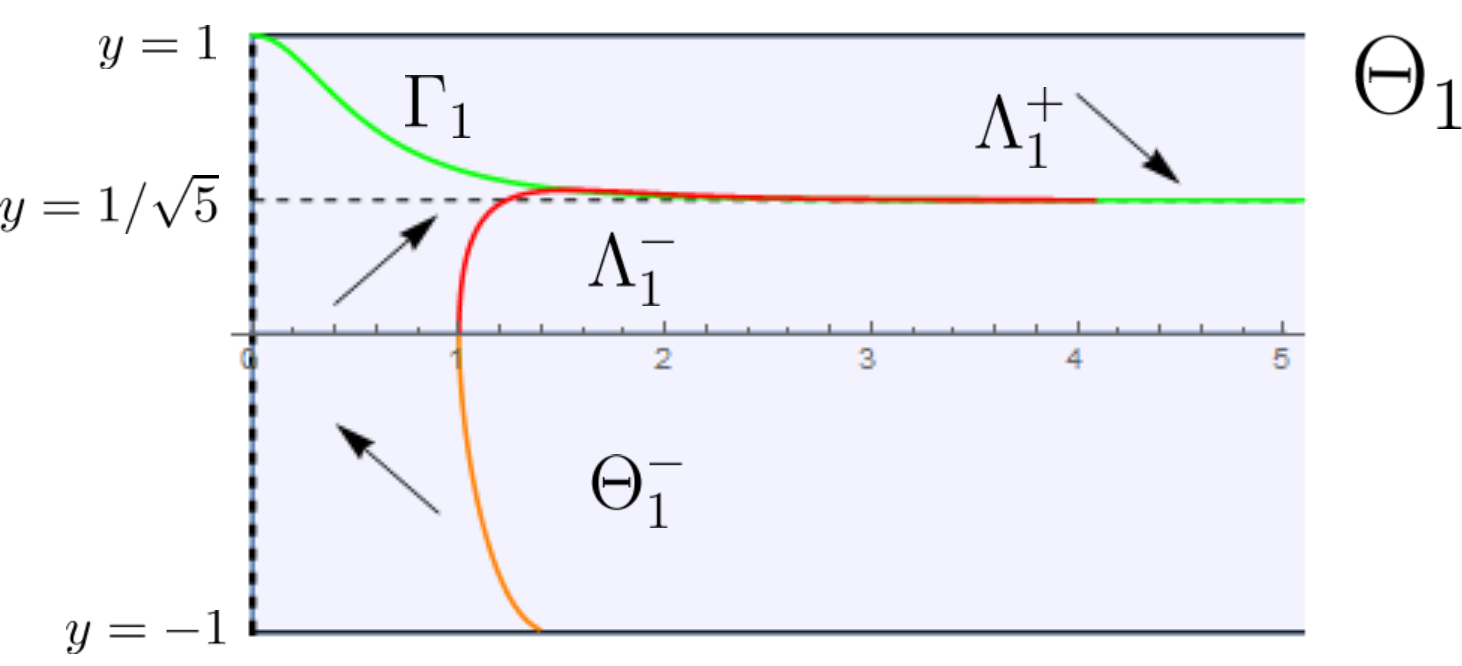}
\caption{Phase space $\Theta_1$. Left, the first component plotted in red. Right, the second component plotted in orange.}\label{fasecat1}
\end{figure}

Now we decrease the parameter $t$ from $t=0$, and the orbit $\gamma$ now lies in the region $\Theta_1^-$ and the coordinate $y(t)$ can be expressed a decreasing graph $y(r)$. Now we let $t$ decrease until $\gamma$ intersects the line $y=-1$ at a point $(r_1,-1)$, where $r_1>r_0$, see Fig. \ref{fasecat1} right. This implies that the generating curve $\alpha(t)$ has a point of horizontal tangency away from the axis of rotation. Then then the phase space changes to $\Theta_{-1}$, and $\gamma$ starts from the point $(r_1,-1)$ contained in $\Lambda_{-1}^+$. Decreasing again $t$, and by Lemma \ref{comportamiento} we ensure that the coordinate $y$ of $\gamma$ can be seen as an increasing graph $y(r)$ that lies entirely in $\Lambda_{-1}^+$ and converges to $y=-1/\sqrt{5}$, see Fig. \ref{fasecat2}, obtaining the second component $M^-$. Similar arguments ensure us that $M^-$ is a graph for all $t<0$, homeomorphic to $\S^1\times [0,\infty)$. For $t\rightarrow -\infty$, the height function $w(t)$ is an increasing function. Again, by properness the height of $M^-$ is unbounded. By uniqueness of the solution of the Cauchy problem for graphs, we can deduce that both components can be smoothly glued together along their planar boundaries, where their unit normals agree, obtaining a complete surface $M$.

\begin{figure}[H]
\centering
\includegraphics[width=0.6\textwidth]{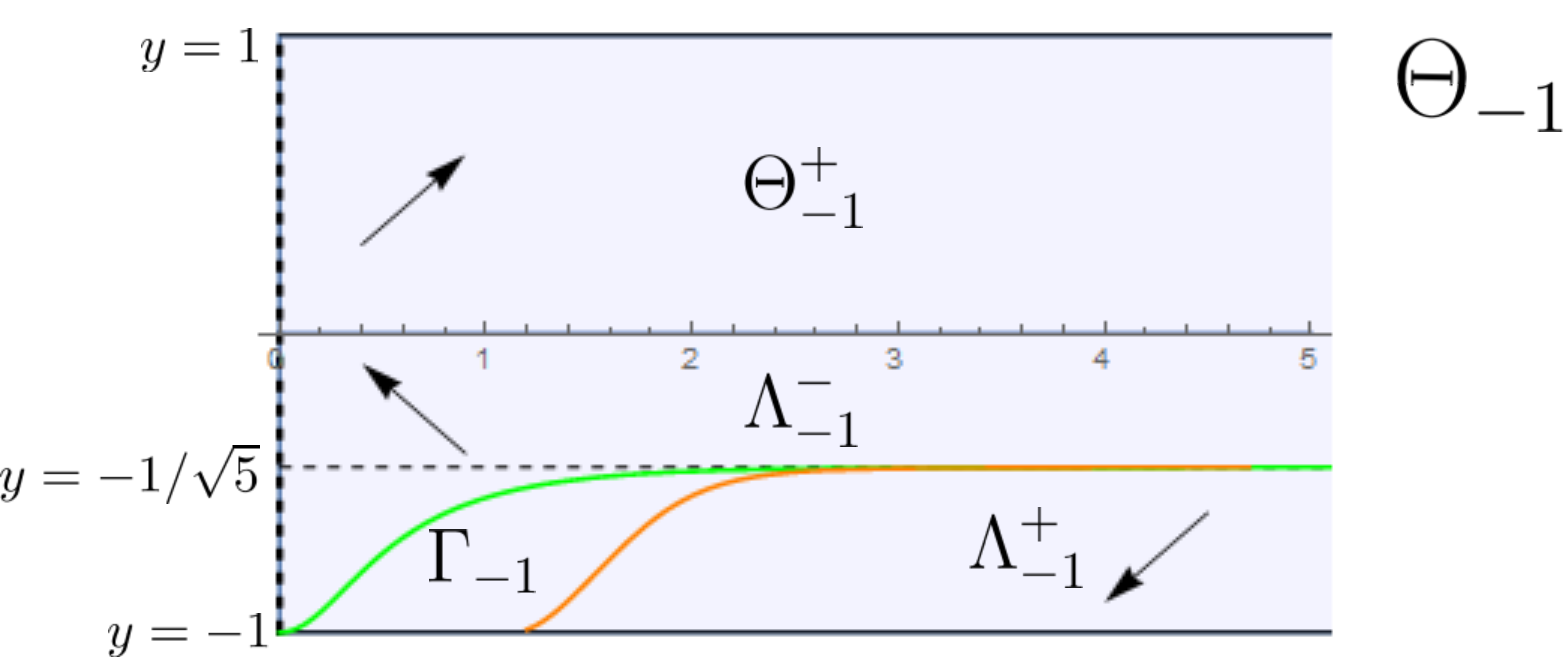}
\caption{Phase space $\Theta_{-1}$ for the second component $M^-$.}\label{fasecat2}
\end{figure}

These examples are the translating catenoids, also known as wing-like solutions. They are characterized as a one parameter family of immersed annuli $\{\cC_r\}_r$, where the parameter $r$ denotes the distance of each $\cC_r$ to the axis of rotation. From the above discussions, for each $r_0>0$ the vertical cylinder $C(0,r_0)$ of radius $r_0$ and centred at the axis of rotation, intersects $\cC_{r_0}$ at an unique circumference with radius $r_0$, which will be called the \emph{neck of the translating catenoid}. Moreover, each $\cC_{r_0}$ lies entirely inside the non-compact component of $\overline{\big(\h2r-C(0,r_0)\big)}$. In Fig. \ref{catenoid} we can see on the left the profile of one translating catenoid, and on the right that  catenoid rotated around the vertical axis passing through the origin, both plotted in the Poincaré disk model of $\h2r$. The points located at the circumference where the minimum height is achieved and have horizontal tangent plane, are those where the phase plane changes from $\Theta_1$ to $\Theta_{-1}$.
\end{proof}

\begin{figure}[H]
\hspace*{3cm}
\includegraphics[width=0.2\textwidth,valign=c]{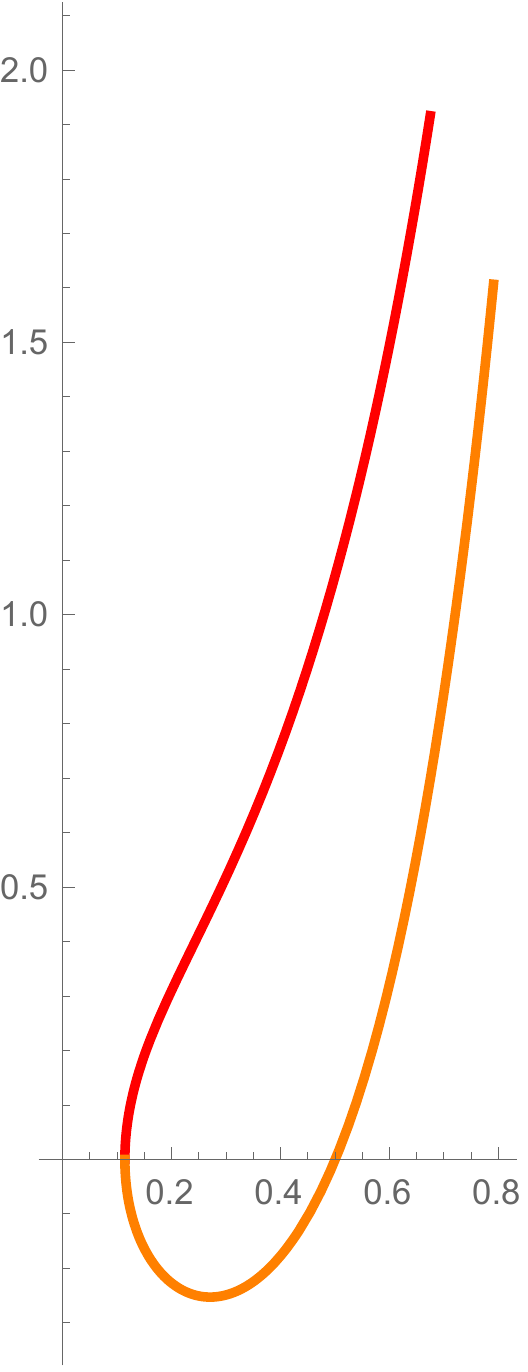}
\includegraphics[width=0.5\textwidth,valign=c]{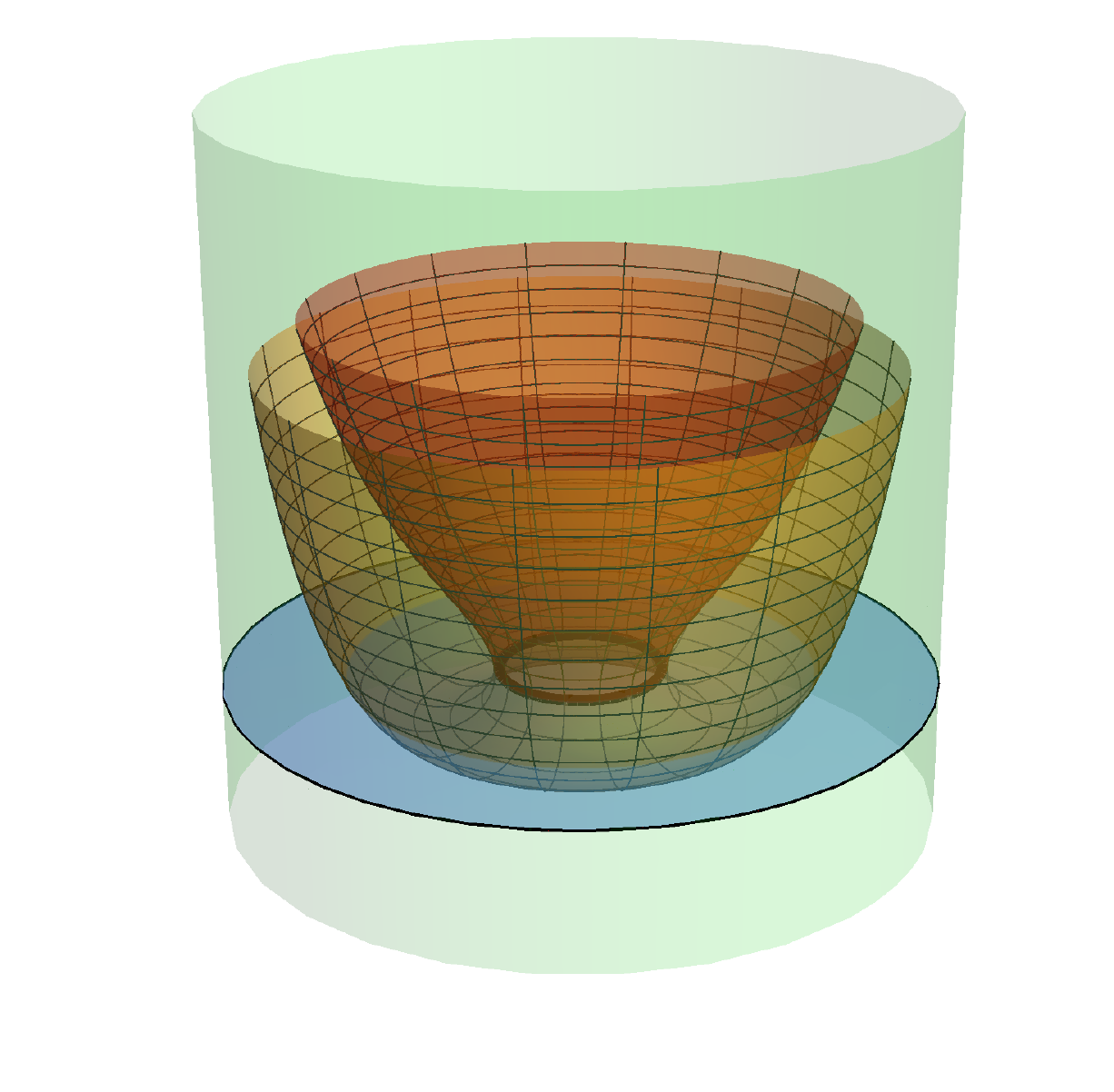}  
\caption{Left: the profile of a translating catenoid, with each component plotted in red and orange, respectively. Right: the translating catenoid in the Poincaré model $\D\times\R$. The neck is plotted in black.}\label{catenoid}
\end{figure}

\section{The asymptotic behavior of the rotational examples}\label{asintotico}

Inspired by the ideas developed in \cite{CSS}, this section is devoted to study the behavior of the bowl soliton $\cB$ and the translating catenoids $\cC_r$ at infinity. 

Let $M$ be a rotational translating soliton in $\h2r$ that is a vertical graph. Such a surface can be parametrized by
\begin{equation}\label{parrotgraph}
\psi(r,\theta)=\big(\sinh r\cos\theta,\sinh r\sin\theta,\cosh r,f(r)\big),\ t\in I,\ \theta\in (0,2\pi),
\end{equation}
for a $C^2$ function $f:I\rightarrow\R$. The angle function $\nu$ of a surface parametrized by Equation \eqref{parrotgraph} is constant in $\theta$, and is given by
\begin{equation}\label{anglerot}
\nu(r)=\frac{1}{\sqrt{1+f'(r)^2}}.
\end{equation}
As $M$ is a translating soliton, the mean curvature of $M$ satisfies $H_M(\psi(r,\theta))=\nu(r)$. This condition writes as
\begin{equation}\label{Hgrafo}
\frac{2}{\sqrt{1+f'(r)^2}}=2H_M(\psi(r,\theta))=\frac{f''(r)}{(1+f'(r)^2)^{3/2}}+\frac{f'(r)}{\sqrt{1+f'(r)^2}}\coth r,
\end{equation}
and after the change $f'(r)=\varphi(r)$,
\begin{equation}\label{cambio}
\varphi'(r)=(1+\varphi(r)^2)(2-\varphi(r)\coth r).
\end{equation}
An exhaustive analysis of Equation \eqref{cambio} allows us to study the asymptotic behavior at infinity of a rotational  soliton.

\begin{lem}\label{asin}
For any $R>0$ and $\varphi_0>0$, there exists a unique smooth solution, $\varphi(r)$ on $[R,\infty)$ to the boundary value problem
\begin{equation}\label{sistemaasin}
\left\{\begin{array}{l}
\varphi'(r)=(1+\varphi(r)^2)(2-\varphi(r)\coth r)\\
\varphi(R)=\varphi_0.
\end{array}\right.
\end{equation}
Moreover, $\lim_{r\rightarrow\infty}\varphi(r)=2$.
\end{lem}
\begin{proof}

The proof is an adaptation of Lemma 2.1 in \cite{CSS} and for the sake of clarity a similar notation will be used. First, notice that fixing $(R,\varphi_0)$ is just fixing initial conditions $(x_0=R,y_0)$, where $y_0=1/\sqrt{1+\varphi^2_0}$, in the phase space $\Theta_1^+$. Thus, existence and uniqueness of the Cauchy problem ensures us the existence of an orbit $\gamma(t)=(r(t),y(t))\subset\Theta_1^+$, with the property that if $s\rightarrow\infty$, then $y(t)\searrow 1/\sqrt{5}$. This gives us the existence of a translating soliton, which is a rotational graph outside a compact set. In particular, the condition $y(t)\searrow 1/\sqrt{5}$ implies that for $s$ big enough the angle function of the solution decreases to the value $1/\sqrt{5}$. 

This translating soliton $M$ can be parametrized by Equation \eqref{parrotgraph} for a $C^2$ function $f:[R,\infty)\rightarrow\R$. In particular, as the angle function $\nu(r)$ is given by Equation \eqref{anglerot}, for $r>r_0$ with $r_0>R$ big enough, $\nu(r)>1/\sqrt{5}$, and thus $\varphi(r)<2$. Moreover, as the angle function is a decreasing function, Equation \eqref{anglerot} implies that $\varphi(r)$ is an increasing function converging to the value $2$. This implies that for $r>r_0$, $\varphi'(r)>0$, and according to Equation \eqref{sistemaasin}, $2-\varphi(r)\coth r$ is positive and remains so. Therefore, we may assume $\varphi(r)\leq 2\tanh r$ for $r>r_0$. In particular, the solution exists for all $r>R$, and by properness of the orbit $\gamma$ the solution cannot become infinite at a finite point.

Now we claim that for every $\varepsilon>0$ and $\Lambda>0$, there exists $r_1>\Lambda$ such that
\begin{equation}\label{mayor}
\varphi(r_1)\geq 2(1-\varepsilon)\tanh(r_1).
\end{equation}
If not, then we substitute the inequality in \eqref{cambio} and obtain for $r>r_0$
$$
\varphi'(r)\geq (1+\varphi(r)^2)2\varepsilon,
$$
which yields after integration
$$
\varphi(r)\geq \tan(2\varepsilon r).
$$ 
For $r$ close enough to $\pi/(4\varepsilon)$ the function $\varphi(r)$ tends to infinity, a contradiction since $\varphi(r)$ is defined for all values of $r>R$.

Let be $\varepsilon>0$ and consider the function $\xi(r)=2(1-\varepsilon)\tanh r$. It is a straightforward fact that the function $\xi(r)$ satisfies
\begin{equation}\label{desig}
\xi'(r)\leq (1+\xi^2(r))\big(2-\xi(r)\coth r\big),
\end{equation}
for $r>r_2$ and $r_2$ sufficiently large. 

Now, Equation \eqref{mayor} ensures us the existence of some $r_3>r_2$ such that $\varphi(r_3)=\xi(r_3)$. Substituting in \eqref{desig} yields  $\xi'(r_3)\leq \varphi'(r_3)$. In this situation we have $\varphi(r)$ and $\xi(r)$ two increasing functions with $\xi(r_3)=\varphi(r_3)$ and  $\varphi'(r)\geq \xi'(r)$ for $r>r_3$. This implies that $\varphi(r)\geq\xi(r)$ for all $r>r_3$. In particular, for $r$ close enough to infinity the function $\varphi(r)$ has the bound
$$
2(1-\varepsilon)\tanh r\leq\varphi(r)\leq 2\tanh r.
$$

Since this is true for every $\varepsilon>0$ and every $r$ large enough, we conclude that $\varphi(r)$ has the asymptotic behavior
\begin{equation}\label{asin1}
\varphi(r)=2\tanh r+o(\tanh r).
\end{equation}

Because $\tanh r$ is a bounded function, the term $o(\tanh r)$ in Equation \eqref{asin1} can be substituted by a negative function $\psi(r)$ tending to zero. Moreover, as we are only interested in asymptotic behavior we can suppose without losing generality that $\psi'(r)>0$ for $r>r_4$ and $r_4$ big enough.

Now let us figure out the asymptotic expression of $\psi(r)$. First of all, observe that because $\psi(r)$ is a negative function tending to zero, then $\psi'(r)$ also tends to zero as $r\rightarrow\infty$; if not, $\psi(r)$ would become positive, a contradiction. If we substitute $\varphi'(r)=2/\cosh^2r+\psi'$ in the ODE given in Equation \eqref{sistemaasin} we get that $\psi'(r)$ satisfies
\begin{equation}\label{deripsi}
\psi'(r)=-\frac{\psi(r)}{\tanh r}\big(1+(2\tanh r+\psi(r))^2\big)-\frac{2}{\cosh^2r}.
\end{equation}
As for $r>r_4$, $\psi'(r)>0$, we get the first bound
\begin{equation}\label{cotainf}
-\psi(r)> \frac{2\tanh r}{\cosh^2 r(1+4\tanh^2r)},\ \forall r>r_4.
\end{equation}

On the other hand, $-\psi(r)$ decreases and tends to zero and thus $-\psi'(r)$ is a negative function tending to zero. This implies that for every $\varepsilon_0>0$, there exists $r_5>0$ large enough such that for every $r>r_5$ we have
$$
-\varepsilon_0<-\psi'(r)=\frac{\psi(r)}{\tanh r}\big(1+(2\tanh r+\psi(r))^2\big)+\frac{2}{\cosh^2r},
$$
which yields
$$
-\psi(r)\big(1+(2\tanh r+\psi(r))^2\big)<\varepsilon_0\tanh r +\frac{2\tanh r}{\cosh^2r}.
$$

As $\psi(r)>-\varepsilon_0$ for $r>r_5$ big enough, we also obtain
$$
-\psi(r)\big(1+(2\tanh r-\varepsilon_0)^2\big)<-\psi(r)\big(1+(2\tanh r+\psi(r))^2\big),
$$
which yields the other bound for $\psi$
\begin{equation}\label{cotasup}
-\psi(r)<\frac{2\tanh r}{\cosh^2 r\big(1+(2\tanh r-\varepsilon_0)^2\big)}+\varepsilon_0\frac{\tanh r}{\big(1+(2\tanh r-\varepsilon_0)^2\big)}.
\end{equation}

As inequality \eqref{cotasup} holds for every $\varepsilon_0>0$, and because $\tanh r/\big(1+(2\tanh r-\varepsilon_0)^2\big)$ is a bounded function, joining \eqref{cotainf} and \eqref{cotasup} we ensure that  $-\psi(r)$ is asymptotic to the function
$$
\frac{2\tanh r}{\cosh^2 r(1+4\tanh^2r)}.
$$

As we defined $\varphi(r)=f'(r)$, we conclude once and for all that a rotational translating soliton that is a graph tending to infinity, is asymptotic to the rotational translating soliton generated by the graph
$$
f(r)=2\log\cosh r+\frac{1}{4}\log\frac{\cosh^2r}{5\cosh(2r)-3},\hspace{.5cm} r>\max\{r_i\}_{i=1,\dots,6}.
$$
On the one hand, the function $\log\cosh r$ satisfies 
$$
\lim_{r\rightarrow\infty}\frac{\log\cosh r}{r}=1.
$$
On the other hand, the function $\log\cosh^2r/(5\cosh(2r)-3)$ is bounded; in fact, its limit is $-\log 10$. Thus, the function $f$ has the asymptotic expansion at infinity
$$
f(r)=2r+k,\ k\in\R.
$$

As Equation \eqref{Hgrafo} is invariant up to additive constants to the function $f(r)$, we conclude that up to vertical translations, the function $f(r)$ has the asymptotic expression $f(r)=2r$ at infinity.
\end{proof}

We want to finish this section by remarking some similarities and differences between the asymptotic behavior that translating solitons and minimal surfaces in $\h2r$ have. Consider the family of translating catenoids $\{\cC_r\}_r$ in $\h2r$. Then, it can be proved that if $r\rightarrow 0$, $\{\cC_r\}_r$ smoothly converges to a double covering of $\cB-\{\textbf{v}\}$, where $\textbf{v}$ is the vertex of $\cB$. 

This also happens in the minimal surface theory in $\h2r$: if consider a vertical axis of rotation $L$, then the minimal surfaces of revolution around $L$ are totally geodesic copies of $\H^2$, and thus minimal, orthogonal to $L$ and a one parameter family of rotationally symmetric properly embedded annuli, the minimal catenoids $\{C_\lambda\}_\lambda$. Here, the parameter $\lambda$ also indicates the distance to $L$, and corresponds to the smallest circumference contained in each catenoid, which is also known as the neck. These minimal catenoids are symmetric bi-graphs over a minimal plane $\Pi$ orthogonal to $L$, and when their neck-sizes converge to zero they converge to a double covering of $\Pi-(\Pi\cap L)$. In this situation, it is natural to relate minimal planes with the  bowl soliton, and minimal catenoids with translating catenoids.

Both the translating catenoids and the minimal catenoids stay at bounded distance to the bowl soliton and the minimal plane, respectively. In particular, this bound on the distance from the translating catenoids to the bowl disables us to apply the same ideas as in \cite{HoMe} in order to obtain \emph{half-space theorems} for properly immersed translating solitons lying at one side of the bowl.

\section{Uniqueness and non-existence theorems for translating solitons}\label{uniqueness}

Most of the results obtained in this section will be proved with the same method: we will use the translating solitons studied in Section \ref{rotacionales} as \emph{canonical surfaces} to compare with, and Theorem \ref{tangency} to arrive to contradictions if some interior tangency point exists. 

\subsection{The uniqueness of the  bowl soliton}

The aim of this first theorem is to give an analogous result to Theorem A in Section 3 in \cite{MSHS}: a complete, embedded translating soliton in the Euclidean space $\R^n$, with a single end smoothly asymptotic to the translating bowl must be a vertical translation of the bowl. Their proof uses Alexandrov reflection technique with respect to vertical planes coming from infinity, and the asymptotic behavior at infinity of the bowl to ensure that Alexandrov reflection technique can start from points close to infinity.

In this section we consider the Poincaré disk model of $\H^2$, as introduced in the beginning of Section \ref{examples}. The hyperbolic distance from a point $(x_1,x_2,0)$ to the origin $(0,0,0)$ will be denoted by $r(x_1,x_2)$

We also give next the following definition:

\begin{defi}
Let $M$ be a properly immersed translating soliton. We say that $M$ is $C^1$-asymptotic to the bowl soliton $\cB$ if for all $\varepsilon>0$ there exists $R>0$ big enough such that $M\cap \big(\h2r-B^3(0,R)\big)$ can be expressed as the graph of a function $g_R:\big(\H^2-B^2(0,R)\big)\rightarrow\R$ such that
\begin{equation}\label{smoothlyasympt}
|g_R(x_1,x_2)-2r(x_1,x_2)|<\varepsilon,\hspace{.3cm} |D_vg_R(x_1,x_2)-2D_vr(x_1,x_2)|<\varepsilon,\  \forall\ r(x_1,x_2)>R,\ \forall\ v\in \H^2,\ |v|=1.
\end{equation}

\end{defi}
Notice that in Lemma \ref{asin1} we already proved that the radial graph defined by the function $f(r)=2r$ converges asymptotically the bowl soliton. Thus, Equation \eqref{smoothlyasympt} implies not only that $M$ converges in distance to $\cB$, but also in its first derivative.

\begin{teo}\label{unicidadbowl}
Let $M$ be a properly immersed translating soliton with a single end, that is $C^1$-asymptotic to the  bowl soliton $\cB$. Then, $M$ is a vertical translation of $\cB$.
\end{teo}

\begin{proof}
Let $M$ be a properly immersed translating soliton with a single end that is $C^1$-asymptotic to the bowl soliton $\cB$. Given a unit vector $v\in\H^2$ and the horizontal direction $(v,0)$ in $\h2r$, the vertical plane orthogonal to $v$ passing through the origin is the totally geodesic surface of $\h2r$ given by the product $\gamma\times\R$, where $\gamma=\gamma(s)\subset\H^2\times\{0\}$ is the arc-length parametrized, horizontal geodesic of $\h2r$ such that $\gamma(0)=(0,0,0)$ and $\gamma'(0)\bot v$. This surface will be denoted by $\Pi_v$. 

Let $\sigma(s)$ be the horizontal geodesic such that $\sigma(0)=(0,0,0)$ and $\sigma'(0)=v$; recall that $\sigma$ and $\gamma$ differ one from the other by a rotation, and thus their arc-length parameter coincides. Consider the 1-parameter family of hyperbolic translations $\{T_s\}$ along the geodesic $\sigma$ such that $T_s(\sigma(0))=\sigma(s)$ for all $s$. We define $\{\Pi_v(s):=T_s(\Pi_v)\}_s$ as the family of vertical planes in $\h2r$ at distance $s$ to $\Pi_v$ and orthogonal to $\sigma'(s)$ at $\sigma(s)$. We denote by $\Pi_v(s)^+$ (resp. $\Pi_v(s)^-$) to the closed half-space $\bigcup_{\lambda\geq s}\Pi_v(\lambda)$ (resp. $\bigcup_{\lambda\leq s}\Pi_v(\lambda)$), and by $M_+(s)$ (resp. $M_-(s)$) to the intersection $M\cap\Pi_v(s)^+$ (resp. $M\cap\Pi_v(s)^-$). The reflection of $M_+(s)$ with respect to the plane $\Pi_v(s)$ will be denoted by $M_+^*(s)$. If $p\in M$ and $p^*\in M^*$ denote the reflected point of $p$ in the reflected surface of $M$, then $\nu^*(p^*)=\nu(p)$, where $\nu^*$ is the angle function of $M^*$. In particular, reflections with respect to vertical planes are isometries of $\h2r$ that send translating solitons into translating solitons. 

Denote by $\mathfrak{p}$ the projection onto the plane $\Pi_v$ defined as follows: Let be $x\in\h2r$ and consider the curve $\alpha_x(s)=T_s(x)$ given as the flow of $T_s$ passing through $x$. Then, we define $\mathfrak{p}(x)$ as the intersection of $\alpha_x(s)$ with the plane $\Pi_v$. This intersection is unique, and thus $\mathfrak{p}(x)$ is well defined. Moreover, after a translation in the arc-length parameter of $\alpha_x(s)$, we will suppose henceforth that $\alpha_x(0)=\mathfrak{p}(x)$.

For a point $x\in\h2r$, let us denote by $I(x)$ to the instant of time $s_0$ such that $\alpha_x(s_0)=x$. We say that $A$ is on the \emph{right hand side} of $B$ if for every $x\in\Pi_v$ such that
$$
\mathfrak{p}^{-1}(x)\cap A\neq\varnothing,\ \ \mathrm{and}\ \ \mathfrak{p}^{-1}(x)\cap B\neq\varnothing,
$$
we have
$$
\inf\big\{I\big(\mathfrak{p}^{-1}(x)\cap A\big) \big\}\geq\sup\big\{I\big(\mathfrak{p}^{-1}(x)\cap B\big) \big\}.
$$
The condition \emph{$A$ is on the right hand side of $B$} is denoted by $A\geq B$, see Figure \ref{proj}.

\begin{figure}[H]
\centering
\includegraphics[width=.85\textwidth]{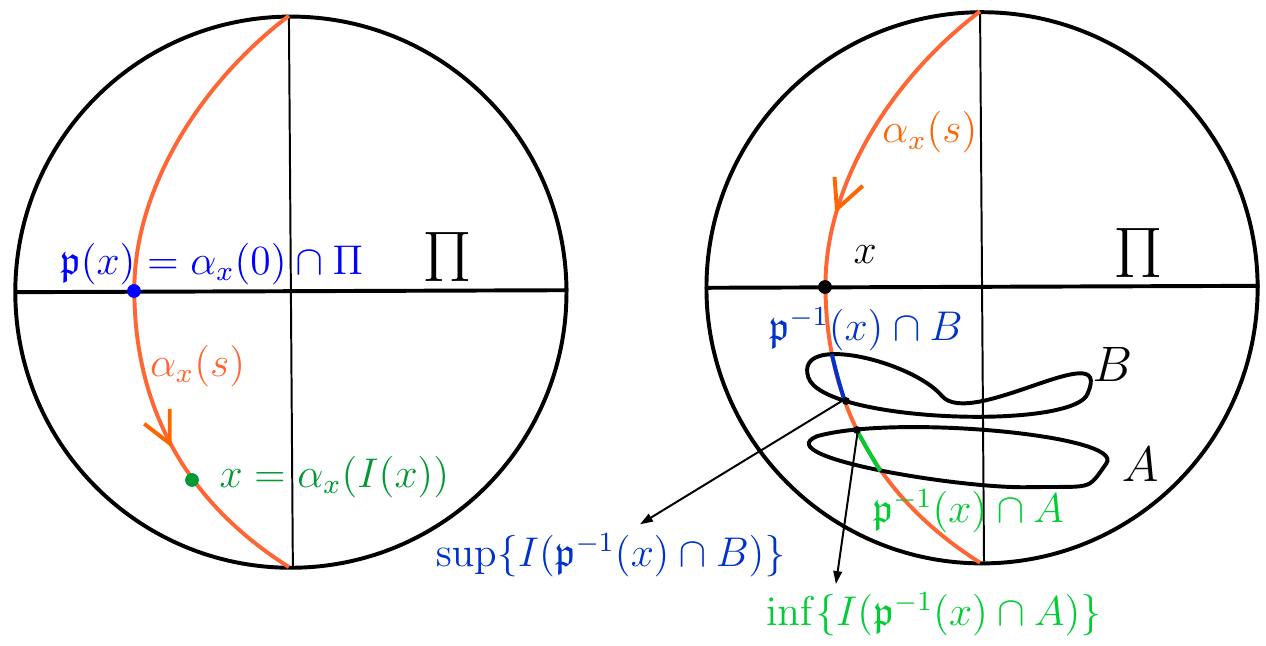}
\caption{Left: the definition of the projection $\mathfrak{p}$. Right: an example of one set $A$ on the right hand side of other set $B$.}\label{proj}
\end{figure}

Now we define the set
$$
\mathcal{A}=\{s\geq 0;\ M_+(s)\ \mathrm{is\ a\ graph\ over}\ \Pi_v,\ \mathrm{and}\ M_+^*(s)\geq M_-(s)\}.
$$
First, we show that $\mathcal{A}\neq\varnothing$. As $M$ is $C^1$-asymptotic to $\mathcal{B}$, for an arbitrary small $\varepsilon>0$ we can choose $s_0>R$ big enough so the intersection $M\cap\{z> s_0\}$ has the topology of an annulus and has distance at most $\varepsilon$ to $\mathcal{B}$. The reflection $M_+^*(s_0)$ is also asymptotic to $\mathcal{B}_+^*(s_0)$ with distance less than $\varepsilon$, and thus does not intersect the surface $M_-(s_0)$ at any interior or boundary point. Moreover, as $\mathcal{B}_+(s_0)$ is a graph onto the plane $\Pi_v$ and $M$ is $C^1$-asymptotic to $\mathcal{B}$, by Equation \eqref{smoothlyasympt} there exists $r_0>R$ such that $D_vg_R (x_1,x_2)>0$, for any $(x_1,x_2)\in\H^2$ such that $r(x_1,x_2)>r_0$. Consequently, and increasing $s_0>r_0$ if necessary, if $\eta$ is a unit normal vector field on $M$ then $\langle\eta,v\rangle>0$ for any point in $M_+(s_0)$. This implies that $M_+(s_0)$ is a graph onto $\Pi_v$.

As $\mathcal{B}_+^*(s_0)$ lies inside the interior domain bounded by $\mathcal{B}$, increasing $s_0$ again if necessary, we can suppose that the distance between $\mathcal{B}_+^*(s_0)$ and $\mathcal{B}_-(s_0)\cap\{z> s_0\}$ is greater than $2\varepsilon$. This implies that $M_+^*(s_0)$ is on the right hand side of $M_-(s_0)$ and thus $s_0\in\mathcal{A}$, proving that $\mathcal{A}$ is a non-empty set. Moreover, we ensure that if $s_0\in\mathcal{A}$, then $[s_0,\infty)\subset\mathcal{A}$. If the assertion fails, then there exists $s_*>s_0$ such that $s_*\notin\mathcal{A}$, and this holds necessary because either $M_+^*(s_*)$ and $M_-(s_*)$ have non-empty intersection, or $M_+(s_*)$ is not a graph onto the plane $\Pi_v$. Thus, there exists $s_1\in(s_0,s_*]$ such that $M_+^*(s_1)$ intersects $M_-(s_1)$ for the first time in an interior point, or $M_+(s_1)$ and $M_+^*(s_1)$ have unit normal agreeing at the boundary. The tangency principle in Theorem \ref{tangency} ensures in its interior or boundary version that $M_+^*(s_1)$ and $M_-(s_1)$ agree, and so the plane $\Pi_v(s_1)$ would be a plane of reflection symmetry of $M$. But $\cB_+^*(s_1)$ and $\cB_-(s_1)$ stay one to each other at a positive distance, since $\Pi_v(s_1)$ is not a plane of reflection symmetry of $\cB$. This is a contradiction to the fact that $M$ is $C^1$-asymptotic to $\cB$, since the reflected component $M_+^*(s_1)$, which agrees with $M_-(s_1)$, stays at positive distance to $\cB_-(s_1)$ . This concludes the proof that $[s_0,\infty)\subset\mathcal{A}$.

The next step is proving that $\mathcal{A}$ is a closed subset of the interval $[0,\infty)$. Indeed, let $\{s_n\}$ be a sequence of points in $\mathcal{A}$ converging to some $s_0$. According to the previous discussion, we have $(s_0,\infty)\subset\mathcal{A}$. First, suppose that $M_+(s_0)$ is not a graph onto the plane $\Pi_v$. Then there exists points $p\neq q\in M_+(s_0)$ such that $\mathfrak{p}(p)=\mathfrak{p}(q)$ and $I(q)>I(p)$. Notice that $I(p)=s_0$, since $s\in\mathcal{A}$ for all $s>s_0$. Let be $I(q)=s_1>s_0$. If we consider the plane $\Pi_v\big({\big(3s_0+s_1\big)\big/4}\big)$, then $M_+^*\big((3s_0+s_1)/4\big)$ cannot be on the right hand side of $M_-\big((3s_0+s_1)/4\big)$ since $I(q^*)<I(p)$, contradicting the fact that $s_0<\big(3s_0+s_1\big)/4\in\mathcal{A}$. The continuity of the graphical condition yields $M_+^*(s_0)\geq M_-(s_0)$ and hence $s_0\in\mathcal{A}$.

Now we will prove that the minimum of the set $\mathcal{A}$ is 0. To prove this, we suppose that $\min\mathcal{A}=s_*>0$, and will arrive to a contradiction. Indeed, if $s_*>0$ then one of the following items must hold:

\begin{itemize}
\item There exists a point $p\in M\cap\Pi_v(s_*)$ such that $M_+(s_*)$ is not a graph at $p$ onto the plane $\Pi_v$. This implies that $\langle\eta_p,v\rangle=0$, where $\eta$ is a unit normal for the surface $M$.

\item There exists a point $p\in M$ such that $M_+^*(s_*)$ and $M_-(s_*)$ have no empty intersection at some $p$, and $M_+^*(s_*)$ lies at the right hand side of $M_-(s_*)$.

\end{itemize}

Notice that in the second item, the intersection between $M_+^*(s_*)$ and $M_-(s_*)$ must be tangential; otherwise for $\varepsilon>0$ small enough $M_+^*(s_*+\varepsilon)$ and $M_-(s_*+\varepsilon)$ would still have a transversal intersection, contradicting that $s_*=\min\mathcal{A}$.

In any case, Theorem \ref{tangency} in its interior or boundary version ensures us that the plane $\Pi_v({s_*})$ is a plane of reflection symmetry of $M$. As the $z$-axis is the axis of rotation of the  bowl soliton $\mathcal{B}$, and every plane of reflection symmetry of $\mathcal{B}$ contains the $z$-axis, the symmetrized $\mathcal{B}_+^*(s_*)$ is on the right hand side of $\mathcal{B}_-(0)$ and lies at a positive distance $d>0$. By hypothesis $M_+^*(s_*)=M_-(s_*)$ has distance to $\mathcal{B}_+^*(s_*)$ tending to zero. Thus, $M_-(s_*)$ has distance to $\mathcal{B}_-(0)$ bounded from below, contradicting the fact that $M$ is $C^1$-asymptotic to $\mathcal{B}$. This implies that $\min\mathcal{A}=0$ and thus we have that $M_+^*(0)\geq M_-(0)$. If we repeat this argument by defining 
$$
\mathcal{A_-}=\{s\leq 0;\ M_-(s)\ \mathrm{is\ a\ graph\ over}\ \Pi_v\ \mathrm{and}\ M_-^*(s)\leq M_+(s)\},	
$$
then we conclude that $M_-^*(0)\leq M_+(0)$. By symmetrizing again we obtain $M_-(0)\geq M_+^*(0)$, and so $M_-(0)=M_+^*(0)$; that is, the plane $\Pi_v$ is a plane of reflection symmetry of the surface $M$. As $v$ was chosen as an arbitrary horizontal vector, we conclude that $M$ is rotationally symmetric around the $z$-axis.  By uniqueness, $M$ is a vertical translation of the  bowl soliton $\mathcal{B}$, completing the proof.
\end{proof}

The following proposition concerning the height function of a translating soliton will be useful:

\begin{pro}\label{hmaximo}
Let $M$ be a compact translating soliton with boundary. Then, the height function of $M$ cannot attain a local maximum in any interior point of $M$.
\end{pro}

\begin{proof}
The proof will be done by contradiction. Suppose that in some $p\in M$, the height function has a local maximum. This implies that there exists a neighbourhood  $U_p$ of $p$ in $M$ such that $h(U_p)\leq h(p)=p_3$, where $h$ is the height function of $M$. In this situation, $U_p$ lies below the horizontal plane $\H^2\times\{p_3\}$. Let $\eta:M\rightarrow\S^2$ be a unit normal vector field to $M$. By hypothesis, $\eta_p=\pm 1$. If $\eta_p=1$, then $M$ has positive mean curvature equal to $1$ at $p$. But $M$ lies locally below $\H^2\times\{p_3\}$ which is a minimal surface and can be oriented upwards without changing the mean curvature. This is a contradiction with the mean curvature comparison principle. If $\eta_p=-1$, we orient $\H^2\times\{p_3\}$ downwards to arrive to the same contradiction.
\end{proof}

Notice that the height function of a translating soliton can achieve a local (or global) minimum, see for example the  bowl soliton or the translating catenoids.

The last theorem in this section has also a counterpart for constant mean curvature surfaces in $\R^3$, and is an important open problem in this theory. It is known that if a compact surface $M\subset\R^3$ with constant mean curvature $H$ and boundary $\partial M$ a circle, lies at one side of the plane $P$ containing the boundary, then $M$ is invariant under rotations around the axis centred at the center of $\partial M$ and orthogonal to $P$, and thus is a part of a sphere of radius $1/H$. However, if the hypothesis on the surface lying at one side of the plane that contains the boundary fails, then the theorem is not known to be true or not. According to Proposition \ref{hmaximo} this cannot happen for translating solitons, and thus compact pieces of the  bowl soliton are unique in the following sense:

\begin{teo}
Let $\Gamma\subset\H^2\times\{t_0\}$ be a closed, embedded curve invariant under rotations around a vertical axis $l$ of $\h2r$. Let $M$ be a compact, embedded translating soliton with boundary $\partial M=\Gamma$. Then, $M$ is rotationally symmetric and, up to translations, is a piece of the  bowl soliton.
\end{teo}

\begin{proof}
The proof will be done by using Alexandrov reflection technique with respect to vertical planes. Without losing generality, after a translation that sends the vertical axis $l$ to the vertical axis $(0,0,t),\ t\in\R$, we may suppose that $\Gamma$ is a circumference centred at the origin with a certain radius. From Proposition \ref{hmaximo}, the translating soliton $M$ lies below the horizontal plane $\H^2\times\{t_0\}$. This consideration is the key that allows us to apply Alexandrov reflection technique, since in the constant mean curvature framework the first contact point may be between an interior and boundary points.

The same notation as in the proof of Theorem \ref{unicidadbowl} will be used here. Let $v$ be an arbitrary, unit horizontal vector, and $\Pi_v$ the vertical plane orthogonal to $v$ passing through the origin. For $s$ big enough, $\Pi_v(s)\cap M=\varnothing$. We start decreasing $s$ until $\Pi_v(s_0)$ intersects $M$ for the first time at a point $p\in M\cap\Pi_v(s_0)$, for some $s_0>0$. Decreasing the parameter $s$, for $s$ close enough to $s_0$ the reflection $M^+_*(s)$ lies inside the interior domain enclosed by $M$. Alexandrov reflection technique stops at some instant $s_1\geq 0$ such that either $M^+_*(s_1)$ is tangent to $M$ at an interior point with the same unit normal; or the intersection between $M^+_*(s_1)$ and $M$ occurs at a boundary point, where their inner conormals agree. In any case, Theorem \ref{tangency} ensures us that the plane $\Pi(s_1)$ is a plane of reflection symmetry of the surface $M$. As the boundary is a circle, the plane $\Pi_v(s_1)$ has to be a plane of reflection symmetry of $\partial M$ as well, and thus $\Pi_v(s_1)$ passes through the center of $\partial M$, which yields $s_1=0$ and $\Pi_v(s_1)\equiv\Pi_v$. Repeating this procedure with all the horizontal directions $v$ we obtain that the surface $M$ is rotationally symmetric around the line passing through the origin, and intersects this axis in an orthogonal way. By uniqueness, $M$ is a compact piece of the translating bowl, as desired.
\end{proof}

\subsection{Non-existence theorems for translating solitons}

In this last section we prove non-existence theorems for translating solitons, assuming some geometric obstructions. The first non-existence result is a straightforward consequence of the divergence theorem, see \cite{Lo}:

\begin{pro}
There do not exist closed (compact without boundary) translating solitons.
\end{pro}

\begin{proof}
Suppose that $M$ is a closed translating soliton, and we will arrive to a contradiction.

It is known that the height function $h:M\rightarrow\R$ on an immersed surface $M$ in $\h2r$ satisfies the PDE $\Delta_M h=2 H_M \nu$, where $\Delta_M$ is the Laplace-Beltrami operator in $M$. As $M$ is a translating soliton, the mean curvature is equal to $H_M=\nu$. Integrating and applying the divergence theorem yields
$$
0=\int_M \Delta_M h=2\int_M \nu^2,
$$
and thus $\nu(p)=0$ for every $p\in M$. This implies that $M$ is contained in a vertical plane, contradicting the fact that $M$ is closed.
\end{proof}

\begin{obs}
The previous Proposition can be also proved by considering the closed  soliton and a vertical plane tangent to the  soliton (such a plane exists by compactness) as minimal surfaces in the conformal space $\big(\h2r,e^h\langle\cdot,\cdot\rangle\big)$, and then arriving to a contradiction by applying Theorem \ref{tangency}. However, for applying this theorem we have to invoke Hopf's maximum principle, which is a more powerful theorem than the divergence theorem.
\end{obs}

Now we prove a height estimate for compact translating solitons with boundary contained in a horizontal plane. Before announcing the result, we will introduce some previous notation that will be useful. Let $\sigma$ be a positive constant. We will denote by $\cB(\sigma)$ to the compact piece of the  bowl soliton that has the circumference $C(0,\sigma)$, centred at the origin and with radius $\sigma$, as boundary. It suffices to intersect $\cB$ with a solid vertical cylinder with axis passing through the origin and radius $\sigma$, and then translate that compact piece in a way that the boundary lies inside the horizontal plane $\H^2\times\{0\}$. The distance from the vertex of $\cB(\sigma)$ to that horizontal plane $\H^2\times\{0\}$ will be denoted by $\tau(\sigma)$. Denote by $\mathcal{Z}_s$ to the flow of the vertical Killing vector field $\partial_z$, which consists on the vertical translations, and let us write by $\cB(\sigma,s)$ to the image of $\cB(\sigma)$ under $\mathcal{Z}_s$.

Also, we state a proposition that has interest on itself, and for the sake of clarity we expose its proof outside the main theorem.

\begin{pro}\label{inslab}
Let $M$ be a compact translating soliton with boundary $\Gamma=\partial M$. If $\Gamma$ lies between two vertical planes, then the whole  soliton $M$ lies between those planes.
\end{pro}

\begin{proof}
Suppose that $M$ has points outside one of the parallel planes, name it $P$. Denote by $P^+$ the component such that $\Gamma\subset P^+$, and by $P^-$ the other component. As $M$ and $\Gamma\subset P^+$, then $M^-=M\cap P^-$ is a compact surface with boundary in $P$. Let $p\in M^-$ be the point with further distance from $M^-$ to $\Pi$. On the one hand, it is clear that $\nu(p)=0$ and thus $H_M(p)=0$. On the other hand, consider vertical planes $P_\lambda$ parallel to $P$ contained in $P^-$ and such that $P_\lambda\cap M^-=\varnothing$. Then we move the parameter $\lambda$ in such a way that the planes $P_\lambda$ move towards $P$ until there exists a first instant $\lambda_0$ such that $P_{\lambda_0}$ intersects $M^-$ precisely at $p$. This contradicts Theorem \ref{tangency} since both $P_{\lambda_0}$ and $M$ are minimal surfaces in the conformal space $\big(\h2r,e^h\langle\cdot,\cdot\rangle\big)$ and thus they should agree, contradicting the fact that $M$ is compact.
\end{proof}

Now we stand in position to formulate the height estimate for compact translating solitons.

\begin{teo}
Let $\Gamma$ be a closed curve of diameter $\sigma$ contained in a horizontal plane $\H^2\times\{t_0\}$ and let $M$ be a compact, connected translating soliton whith boundary $\partial M=\Gamma$. Then, for all $p\in M$, the distance from $p$ to $\H^2\times\{t_0\}$ is less or equal than $\tau(2\sigma)$, where $\tau(\sigma)$ is the constant defined above.
\end{teo}

\begin{proof}
Let $\sigma$ be the diameter of $\Gamma$. We can apply a translation $T$ to $\Gamma$ such that $T\big(\Gamma\big)$ lies inside the disk $D\big(0,2\sigma\big)$. For saving notation we will just denote $T\big(\Gamma\big)$ by $\Gamma$. Consider now the compact piece $\cB(2\sigma)$. By Proposition \ref{hmaximo} we know that $M$ lies below the horizontal plane $\H^2\times\{0\}$. As $\Gamma\subset D\big(0,2\sigma\big)$, initially the surface $M$ is inside the mean convex region enclosed by $\cB(2\sigma)$. We assert that the entire surface $M$ lies strictly in this region. Indeed, suppose that $\cB(2\sigma)$ and $M$ have non-empty intersection. This intersection has to be transversal since otherwise we would have by Theorem \ref{tangency} that $M=\cB(2\sigma)$ and thus $\Gamma=C\big(0,2\sigma\big)$, contradicting the fact that $\Gamma$ has diameter $\sigma$. Translate downwards the graph $\cB(2\sigma)$ until $\cB(2\sigma,t_0)\cap M=\varnothing$, for $t_0$ small enough. Then move upwards $\cB(2\sigma,t_0)$ by increasing $t_0$ until we reach a first contact point, which has to be an interior point at some horizontal plane $\H^2\times\{t_1\}$. Then, Theorem \ref{tangency} ensures us that $M$ and $\cB(2\sigma,t_1)$ must agree, and thus $\Gamma=C\big((0,0,t_1),2\sigma\big)$. But in the instant of time that both surfaces coincide, their boundaries are in different planes, which is absurd. Thus, $M$ lies inside the mean convex side of $\cB(2\sigma,0)$ and we obtain the desired height estimate.
\end{proof}

The two last theorems give geometric obstructions for the existence of certain translating solitons.

\begin{teo}
There do not exist properly immersed translating solitons in $\h2r$ contained inside a compact vertical cylinder.
\end{teo}

\begin{proof}
Suppose that $M$ is a properly immersed translating soliton lying inside a vertical cylinder of radius $r_0$, and denote it by $C(r_0)$. After a translation we can suppose that the axis of the cylinder is the straight line $l=(0,0,t),\ t\in\R$. Consider the family of translating catenoids $\{\cC_r\}_r$ rotated around $l$. For $r>r_0$ each $\cC_r$ lie inside the non-compact component of $\overline{\h2r-C(r_0)}$. Now we start decreasing the parameter $r$ until we reach an interior tangency point between $M$ and $\cC_{r_1}$, for some $r_1\leq r_0$. This is a contradiction with Theorem \ref{tangency} since $M$ and $\cC_{r_1}$ would agree, but none $\cC_r$ lies inside a vertical cylinder.
\end{proof}

\noindent The author was partially supported by MICINN-FEDER, Grant No. MTM2016-80313-P and Junta de Andalucía Grant No. FQM325.

\begin{thebibliography}{19}

\bibitem[1]{AlWu} S. Altschuler, L. Wu, Translating surfaces of the non-parametric mean curvature flow with prescribed contact angle, {\it Calc. Var. Partial Differential Equations} {\bf 2} (1994), no. 1, 101--111

\bibitem[2]{Ba} V. Bayle, Propriétés de concavité du profil isopérimétrique et applications. Ph.D. Thesis,  Institut Joseph Fourier, Grenoble, 2003.

\bibitem[3]{BGM} A. Bueno, J.A. Gálvez, P. Mira, The global geometry of surfaces with prescribed mean curvature in $\R^3$, preprint, arXiv:1802.08146.

\bibitem[4]{BCMR} V. Bayle, A. Cañete, F. Morgan, C. Rosales, On the isoperimetric problem in Euclidean space with density, {\it Calc. Var. Partial Differential Equations} {\bf 31} (2008), 27--46.

\bibitem[5]{CHH} J.-B. Casteras, E. Heinonen, I. Holopainen, Dirichlet problem for $f-$minimal graphs, Preprint arXiv:1605.01935.

\bibitem[6]{CSS} J. Clutterbuck, O. Schnurer, and F. Schulze, Stability of translating solutions to mean curvature flow, {\it  Calc. Var. Partial Differential Equations} {\bf 29} (2007), no. 3, 281--293.

\bibitem[7]{GT} D. Gilbarg and N. S. Trudinger, Elliptic partial differential equations of second order, {\it Classics in Mathematics, Springer-Verlag, Berlin} (2001), Reprint of the 1998 edition.

\bibitem[8]{Gr} M. Gromov, Isoperimetry of waists and concentration of maps, {\it Geom. Funct. Anal.} {\bf 13} (2003), 178--215.

\bibitem[9]{HoMe} D. Hoffman, W.H. Meeks III, The strong halfspace theorem for minimal surfaces, {\it Invent. Math.} {\bf 101} (1990), 373--377.

\bibitem[10]{Hu} G. Huisken, The volume preserving mean curvature flow, {\it J. Reine Angew. Math.} {\bf 382} (1987), 35--48.

\bibitem[11]{HuSi} G. Huisken, C. Sinestrari, Convexity estimates for mean curvature flow and singularities of mean convex surfaces, {\it Acta Mathematica} {\bf 183} (1993), no. 1, 45--70.

\bibitem[12]{Il} T. Ilmanen, Elliptic regularization and partial regularity for motion by mean curvature, {\it Mem. Amer. Math. Soc.} {\bf 108} (1994), no. 520.

\bibitem[13]{KoOr} Kocakusakli, E. and Ortega, M., Extending Translating Solitons in Semi-Riemannian Manifolds, \emph{Lorentzian Geometry and Related Topics}, Springer Proceedings in Mathematics and Statistics \textbf{211} (2016).

\bibitem[14]{LaOr} Lawn, M. A. and Ortega, M., Translating Solitons From Semi-Riemannian Submersions, preprint, arXiv.1607.04571.

\bibitem[15]{Lo} R. López, Invariant surfaces in Euclidean space with a log-linear density, preprint, arXiv:1802.07987.

\bibitem[16]{MPGSHS} F. Martín, J. Pérez-García, A. Savas-Halilaj, and K. Smoczyk, A characterization of the grim reaper cylinder, arXiv:1508.01539. To appear in {\it Journal fur die reine und angewandte Mathematik.} 

\bibitem[17]{MSHS} F. Martín, A. Savas-Halilaj, and K. Smoczyk, On the topology of translating solitons of the mean curvature flow, {\it Calculus of Variations and Partial Differential Equations} {\bf 54} (2015), no. 3, 2853-2882.

\bibitem[18]{NeRo1} B. Nelli, H. Rosenberg, Simply connected constant mean curvature surfaces in $\mathbb{H}^2\times\R$, \emph{Michigan Math. J.} \textbf{54}, Issue 3 (2006), 537--544.

\bibitem[19]{NeRo2} B. Nelli, H. Rosenberg, Global properties of Constant Mean Curvature surfaces in $\mathbb{H}^2\times\R$, \emph{ Pac. Jour. Math.} \textbf{226}, no. 1 (2006), 137--152.

\bibitem[20]{Pe} J. Pérez, Translating solitons of the mean curvature flow, Ph.D. Thesis, Universidad de Granada 2016.

\bibitem[21]{Sm} G. Smith, On complete embedded translating solitions of the mean curvature flow that are of finite genus, preprint, arXiv:1501.04149.

\bibitem[22]{SpXi} J. Spruck, L. Xiao, Complete translating solitons to the mean curvature flow in $\R^3$ with nonnegative mean curvature, preprint, arXiv:1703.01003.

\end{thebibliography}
\end{document}